\documentclass[a4paper, 11pt]{article}

\usepackage{amsmath, amssymb, amsthm} % AMS series
\usepackage{color}
\usepackage[top=40truemm,bottom=30truemm,left=30truemm,right=30truemm]{geometry}
\usepackage{hyperref}
\usepackage[all]{xy}
\usepackage{enumitem}

\DeclareMathOperator{\bd}{bd}
\DeclareMathOperator{\cls}{cl}
\DeclareMathOperator*{\Limsup}{Lim\,sup}
\DeclareMathOperator*{\Liminf}{Lim\,inf}

\newcommand{\amd}{\mspace{10mu} \text{and} \mspace{10mu}}
\newcommand{\argdot}{\mspace{2mu} \cdot \mspace{2mu}}
\newcommand{\ep}{\varepsilon}
\newcommand{\R}{\mathbb{R}}
\newcommand{\Set}[2]{\left\{ \mspace{1mu} #1 : #2 \mspace{1mu} \right\}}
\newcommand{\tto}{\rightrightarrows}
\newcommand{\Z}{\mathbb{Z}}

\theoremstyle{plain}
\newtheorem{theorem}{Theorem}[section]

\newtheorem{lemma}[theorem]{Lemma}

\newtheorem{corollary}[theorem]{Corollary}

\theoremstyle{definition}
\newtheorem{definition}[theorem]{Definition}

\theoremstyle{remark}
\newtheorem{remark}[theorem]{Remark}
\newtheorem{notation}{Notation}

\numberwithin{equation}{section}

\begin{document}

\title{Asymptotic compactness in topological spaces}
\author{Junya Nishiguchi\thanks{Mathematical Science Group, Advanced Institute for Materials Research (AIMR), Tohoku University,
2-1-1 Katahira, Aoba-ku, Sendai, 980-8577, Japan}
\footnote{E-mail: \url{junya.nishiguchi.b1@tohoku.ac.jp}}}
\date{}

\maketitle

\begin{abstract}
The omega limit sets plays a fundamental role to construct global attractors for topological semi-dynamical systems with continuous time or discrete time.
Therefore, it is important to know when omega limit sets become nonempty compact sets.
The purpose of this paper is to understand the mechanism under which a given net of subsets of topological spaces is compact in the asymptotic sense.
For this purpose, we introduce the notion of asymptotic compactness for nets of subsets and study the connection with the compactness of the limit sets.
In this paper, for a given net of nonempty subsets, we prove that the asymptotic compactness and the property that the limit set is a nonempty compact set to which the net converges from above are equivalent in uniformizable spaces.
We also study the sequential version of the notion of asymptotic compactness by introducing the notion of sequentiality of directed sets.

\begin{flushleft}
\textbf{2020 Mathematics Subject Classification}.
Primary 54A20, 54D30, 54E15; Secondary 37B02, 54C60
\end{flushleft}

\begin{flushleft}
\textbf{Keywords}.
Limit sets; compactness; uniformizable spaces; general topology; set-valued analysis.
\end{flushleft}

\end{abstract}

\tableofcontents

\section{Introduction}

In the theory of topological semi-dynamical systems with continuous time or discrete time, the omega limit sets is a one of the fundamental objects for the asymptotic behavior of the orbits of the semi-dynamical systems (ref.\ Gottschalk and Hedlund~\cite[Section 10]{Gottschalk--Hedlund 1955}).
The omega limit sets also plays a fundamental role in the construction of global attractors in infinite-dimensional semi-dynamical systems (refs.\ Hale~\cite{Hale 1988}, Sell and You~\cite[Chapter 2]{Sell--You 2002}, and Raugel~\cite{Raugel 2002}).
Therefore, it is important to understand when omega limit sets are nonempty compact sets when the phase space is neither compact nor locally compact.

A related notion for the non-emptiness and the compactness of omega limit sets is known in the literature.
It is the \textit{asymptotic compactness} introduced in \cite{Sell--You 2002} when the phase space of semi-dynamical systems is a metric space.
This notion of asymptotic compactness is expressed by sequences.
Therefore, it is not apparent how this notion should be generalized to the case that the phase space is a general topological space because omega limit points (i.e., points belonging to omega limit sets) are not expected to be expressed by sequences.

In this paper, we will tackle this problem by an approach considering nets $(X_s)_{s \in S}$ of subsets of a topological space $X$ for some directed set $S = (S, \le)$.
Then the \textit{limit set} $L(X_s)_{s \in S}$ is defined by
\begin{equation*}
	L(X_s)_{s \in S} = \bigcap_{s \in S} \cls \bigcup_{t \in S, t \ge s} X_t.
\end{equation*}
Here $\cls$ denotes the closure operator in the topological space $X$.
For the above mentioned problem of omega limit sets of topological semi-dynamical systems with continuous time or discrete time, the directed set $S$ corresponds to the set of nonnegative reals $\R^+$ or the set of nonnegative integers $\Z^+$, respectively.
Then the \textit{omega limit set} $\omega_\varPhi(E)$ of a subset $E \subset X$ for a semiflow $\varPhi \colon S \times X \to X$ is given by the limit set of the net $(\varPhi(\{s\} \times E))_{s \in S}$ of subsets of $X$.
Here a \textit{semiflow} $\varPhi$ is a map with the properties that (i) $\varPhi(0, x) = x$ for all $x \in X$ and (ii) $\varPhi(t + s, x) = \varPhi(t, \varPhi(s, x))$ for all $t, s \in S$ and all $x \in X$.

This approach is general from the perspective on considering omega limit sets of topological semi-dynamical systems with continuous time or discrete time.
At the same time, this approach should be natural because the nets $(X_s)_{s \in S}$ contains nets of points in $X$, which are necessary to consider the convergence concept in general topological spaces.
Furthermore, this approach permits us to study omega limit sets of topological semi-dynamical systems with a \textit{preordered abelian phase group} $T = (T, +, \le)$ (i.e., an abelian phase group with translation-invariant preorder) because the positive cone given by $T^+ := \Set{t \in T}{t \ge 0}$ becomes a directed set.
For example, see also \cite{Hahn 1961} for a study of this direction.

For the net $(X_s)_{s \in S}$ of subsets of $X$, we will obtain characterizations of points belonging to the limit set $L(X_s)_{s \in S}$.
Then we will define the notions of \textit{asymptotic compactness} and \textit{weak asymptotic compactness} of the net $(X_s)_{s \in S}$ based on these characterizations.
See Definitions~\ref{dfn:asymptotic cptness} and \ref{dfn:weak asymptotic cptness} for further details.
These are the key notions of this paper, which are expected to be related with the compactness of the limit set $L(X_s)_{s \in S}$.
We will also introduce the convergence concepts for the net $(X_s)_{s \in S}$ called the \textit{convergence from above} and the \textit{convergence from below} based on the upper semicontinuity and the lower semicontinuity of set-valued maps because $(X_s)_{s \in S}$ is a set-valued map from $S$ to $X$.
See Definitions~\ref{dfn:convergence from above of a net of subsets} and \ref{dfn:convergence from below of a net of subsets} for their definitions.
To make clear the above mentioned connection, we will introduce the terminology of the \textit{limit set compactness} of $(X_s)_{s \in S}$, which means that $L(X_s)_{s \in S}$ is a nonempty compact set to which $(X_s)_{s \in S}$ converges from above.

Unfortunately, it seems to be optimistic to expect that the asymptotic compactness and the limit set compactness are equivalent in any general topological space.
One of the reason is that even if we choose a net $(z_\alpha)_{\alpha \in A}$ for some directed set $A$ in the limit set $L(X_s)_{s \in S}$, we cannot associate this net with a net in general which is related to the net $(X_s)_{s \in S}$.
Then we cannot extract the full power of the asymptotic compactness.

To overcome this difficulty, we rely on the uniformizability of the topological space $X$.
Under the assumption of the uniformizabiity, for a given net $(z_\alpha)_{\alpha \in A}$ in $L(X_s)_{s \in S}$, we can choose a net $(y_\beta)_{\beta \in B}$ for some directed set $B$ which is related to the net $(X_s)_{s \in S}$ so that $z_\alpha$ and $y_\beta$ are in some uniform nearness.
Here the uniformizability is essentially used.
The mathematically precise statement is given in the proof of Theorem~\ref{thm:asymptotic cptness and limit set cptness in uniformizable sp}.

The final task of this paper is to clarify the relation between the asymptotic compactness of the net $(X_s)_{s \in S}$ and the asymptotic compactness for semiflows with continuous time in metric spaces introduced in \cite{Sell--You 2002}.
For this purpose, we will introduce the sequential versions of the asymptotic compactness and the weak asymptotic compactness (see Definitions~\ref{dfn:asymptotic seq cptness} and \ref{dfn:weak asymptotic seq cptness}).
To introduce the sequential versions, we need to restrict a class of directed sets.
This is a class of \textit{sequential directed sets}, in which a directed set has a sequence $(s_n)_{n = 1}^\infty$ which becomes larger and larger as $n \to \infty$.
See Definition~\ref{dfn:sequential directed set} for the precise definition.
Then we will obtain the equivalence between the asymptotic compactness and the asymptotic sequential compactness for the net $(X_s)_{s \in S}$ when $S$ is sequential and the topological space $X$ is pseudo-metrizable.
This shows the above mentioned equivalence.

This paper is organized as follows.
In Section~\ref{sec:limit sets and convergence}, we study the limit set and the convergence property for a given net $(X_s)_{s \in S}$ of subsets of $X$ for some directed set $S$.
Then we obtain characterizations of points belonging to the limit set $L(X_s)_{s \in S}$ in Theorem~\ref{thm:characterization of limit set}.
We also study various connections between the convergence of the net $(X_s)_{s \in S}$ and the limit set $L(X_s)_{s \in S}$.
In Section~\ref{sec:asymptotic cptness and limit set cptness}, we introduce the notions of asymptotic compactness and weak asymptotic compactness of the net $(X_s)_{s \in S}$ of subsets in Definitions~\ref{dfn:asymptotic cptness} and \ref{dfn:weak asymptotic cptness}.
In Lemma~\ref{lem:eventual Lagrange stability and limit set cptness in locally cpt regular sp}, we reveal the connection between the asymptotic compactness and the limit set compactness of $(X_s)_{s \in S}$ when $X$ is a locally compact regular space by using the notion of eventual Lagrange stability.
In Section~\ref{sec:asymptotic cptness and limit set cptness in uniformizable spaces}, we investigate the asymptotic compactness of nets of subsets in uniformizable spaces.
One of the main result of this paper is Theorem~\ref{thm:equivalence between asymptotic cptness and limit set cptness in uniformizable sp}, which shows that the following properties are equivalent when $X$ is uniformizable: (i) $(X_s)_{s \in S}$ converges from above to some nonempty compact set, (ii) $(X_s)_{s \in S}$ is asymptotically compact, and (iii) $(X_s)_{s \in S}$ is limit set compact.
In Section~\ref{sec:asymptotic seq cptness}, we study the sequential versions of the asymptotic compactness and the weak asymptotic compactness.
Under the assumption that $S$ is sequential and $X$ is pseudo-metrizable, we obtain the sequential version of Theorem~\ref{thm:equivalence between asymptotic cptness and limit set cptness in uniformizable sp} in Theorem~\ref{thm:asymptotic seq cptness and limit set cptness in pseudo-metrizable sp}.
By combining these theorems, we finally obtain the equivalence between the asymptotic compactness and the asymptotic sequential compactness in Corollary~\ref{cor:asymptotic cptness and asymptotic seq cptness in pseudo-metrizable sp}.

\section{Limit sets and convergence of nets of subsets}\label{sec:limit sets and convergence}

Throughout this section, let $X$ be a topological space.

\subsection{Limit sets}

In this subsection, we will investigate characterizations of points belonging to the limit set of a net of subsets of $X$.
For this purpose, we first recall the convergence concepts in topological spaces via nets.
We refer the reader to \cite{Kelley 1955} as a general reference of general topology.

\subsubsection{Convergence in topological spaces}

\begin{definition}[ref.\ \cite{Kelley 1955}]\label{dfn:directed set}
A nonempty set $A$ together with a preorder $\le$ on $A$ (i.e., a binary relation on $A$ with the reflexivity and the transitivity) is called a \textit{directed set} if every pair of $a, b \in A$ has an upper bound $c \in A$, i.e., an element $c \in A$ satisfying $a \le c$ and $b \le c$.
A directed set $(A, \le)$ is called a \textit{directed poset} if $\le$ is a partial order, i.e., $\le$ satisfies the antisymmetry.
\end{definition}

\begin{remark}
For each directed sets $A = (A, \le)$ and $B = (B, \le)$, the Cartesian product $A \times B$ is considered as a directed set with the preorder $\le$ defined as follows: $(\alpha_1, \beta_1) \le (\alpha_2, \beta_2)$ if $\alpha_1 \le \alpha_2$ and $\beta_1 \le \beta_2$.
The directed set $(A \times B, \le)$ is called the \textit{product directed set}.
\end{remark}

\begin{definition}[ref.\ \cite{Kelley 1955}]\label{dfn:nets and sequences}
A family $(x_\alpha)_{\alpha \in A}$ in some set for some directed set $A$ is called a \textit{net}.
\end{definition}

\begin{definition}[ref.\ \cite{Kelley 1955}]
Let $A = (A, \le)$ be a directed set and $U \subset X$ be a subset.
\begin{itemize}
\item A net $(x_\alpha)_{\alpha \in A}$ in $X$ is said to be \textit{eventually} in $U$ if there exists $\alpha_0 \in A$ such that for all $\alpha \in A$, $\alpha \ge \alpha_0$ implies $x_\alpha \in U$.
\item A net $(x_\alpha)_{\alpha \in A}$ in $X$ is said to be \textit{frequently} in $U$ if for every $\alpha \in A$, there exists $\beta \in A$ such that $\beta \ge \alpha$ and $x_\beta \in U$.
\end{itemize}
\end{definition}

\begin{definition}[ref.\ \cite{Kelley 1955}]\label{dfn:convergence and cluster pt}
Let $A = (A, \le)$ be a directed set and $x \in X$ be given.
\begin{itemize}
\item A net $(x_\alpha)_{\alpha \in A}$ in $X$ is said to \textit{converge} to $x$ if $(x_\alpha)_{\alpha \in A}$ eventually in every neighborhood of $x$.
\item A net $(x_\alpha)_{\alpha \in A}$ in $X$ is said to have a \textit{cluster point} $x$ if $(x_\alpha)_{\alpha \in A}$ frequently in every neighborhood of $x$.
\end{itemize}
\end{definition}

\begin{definition}[ref.\ \cite{Kelley 1955}]\label{dfn:cofinality}
Let $A = (A, \le)$ be a directed set.
A subset $S \subset A$ is said to be \textit{cofinal} if for every $\alpha \in A$, there exists $\alpha' \in S$ such that $\alpha' \ge \alpha$.
\end{definition}

\begin{remark}
A subset $S$ of the directed set $A$ is not necessarily directed.
We note that any cofinal subset of $A$ is directed.
\end{remark}

\begin{definition}\label{dfn:finality}
Let $A = (A, \le)$ be a directed set and $B$ be a set.
A map $h \colon B \to A$ is said to be \textit{final} if the image $h(B) := \Set{h(\beta)}{\beta \in B}$ is a cofinal subset of $A$.
\end{definition}

In this paper, we adopt the following notation.

\begin{notation}
Let $A = (A, \le), B = (B, \le)$ be directed sets and $h \colon B \to A$ be a map.
By $h(\beta) \to \bd(A)$ as $\beta \to \bd(B)$, we mean that for every $\alpha \in A$, there exists $\beta_0 \in B$ such that for all $\beta \in B$, $\beta \ge \beta_0$ implies $h(\beta) \ge \alpha$.
\end{notation}

We note that a map $h \colon A \to A$ satisfying $h(\alpha) \ge \alpha$ for all $\alpha \in A$ has the property $h(\alpha) \to \bd(A)$ as $\alpha \to \bd(A)$.

\begin{definition}[ref.\ \cite{Kelley 1955}]\label{dfn:subnet}
Let $(x_\alpha)_{\alpha \in A}$ be a net in $X$ for some directed set $A$.
For every directed set $B$, a net $(y_\beta)_{\beta \in B}$ is called a \textit{subnet} of $(x_\alpha)_{\alpha \in A}$ if there exists a map $h \colon B \to A$ such that (i) $h(\beta) \to \bd(A)$ as $\beta \to \bd(B)$ and (ii) $y_\beta = x_{h(\beta)}$ holds for all $\beta \in B$.
\end{definition}

In this paper, we do not adopt the convention that the map $h \colon B \to A$ is monotone and final for a subnet $(x_{h(\beta)})_{\beta \in B}$.
The following theorem gives characterizations of cluster points of nets.
See \cite{Kelley 1955} for the proof.
See also Theorem~\ref{thm:characterization of limit set} for an extension of this characterizations to nets of subsets.

\begin{theorem}[ref.\ \cite{Kelley 1955}]\label{thm:cluster pt and convergent subnet}
Let $A$ be a directed set, $(x_\alpha)_{\alpha \in A}$ be a net in $X$, and $x \in X$ be given.
Then the following properties are equivalent:
\begin{enumerate}[label=\textup{(\alph*)}]
\item $x$ is a cluster point of $(x_\alpha)_{\alpha \in A}$.
\item There exist a directed set $B$ and a monotone final map $h \colon B \to A$ such that $(x_{h(\beta)})_{\beta \in B}$ converges to $x$.
\item There exist a directed set $B$ and a map $h \colon B \to A$ such that $h(\beta) \to \bd(A)$ as $\beta \to \bd(B)$ and $(x_{h(\beta)})_{\beta \in B}$ converges to $x$.
\end{enumerate}
\end{theorem}

\begin{notation}
For each $x \in X$, let $\mathcal{N}_x$ denote the set of all neighborhoods of $x$.
It is considered to be a directed poset with the partial order $\le$ defined as follows:
For all $U_1, U_2 \in \mathcal{N}_x$, $U_1 \le U_2$ if $U_1 \supset U_2$.
\end{notation}

\begin{notation}
For each subset $E \subset X$, let $\cls(E)$ denote the closure of $E$ (i.e., the smallest closed set of $X$ containing $E$).
Then $x \in \cls(E)$ if and only of $U \cap E \ne \emptyset$ holds for every $U \in \mathcal{N}_x$.
\end{notation}

\begin{remark}\label{rmk:set of cluster pts}
Let $A = (A, \le)$ be a directed set and $(x_\alpha)_{\alpha \in A}$ be a net in $X$.
Then $x \in \bigcap_{\alpha \in A} \cls \Set{x_\beta}{\beta \in A, \beta \ge \alpha}$ if and only if for every $\alpha \in A$ and every $U \in \mathcal{N}_x$,
\begin{equation*}
	U \cap \Set{x_\beta}{\beta \in A, \beta \ge \alpha} \ne \emptyset.
\end{equation*}
Therefore,
\begin{equation*}
	\bigcap_{\alpha \in A} \cls \Set{x_\beta}{\beta \in A, \beta \ge \alpha}
\end{equation*}
is equal to the set of cluster points of $(x_\alpha)_{\alpha \in A}$.
\end{remark}

The following is another characterization of the points belonging to the closure of a nonempty subset.

\begin{lemma}\label{lem:closure and net}
Let $E \subset X$ be a nonempty subset and $x \in X$ be given.
Then $x \in \cls(E)$ if and only if there exist a directed set $A$ and a net $(x_\alpha)_{\alpha \in A}$ in $E$ such that $(x_\alpha)_{\alpha \in A}$ converges to $x$.
\end{lemma}

The proof is standard, and therefore, it can be omitted.

\subsubsection{Limit sets and their characterizations}

\begin{definition}\label{dfn:limit set}
Let $S = (S, \le)$ be a directed set and $(X_s)_{s \in S}$ be a net of subsets of $X$.
The subset $L(X_s)_{s \in S} \subset X$ defined by
\begin{equation*}
	L(X_s)_{s \in S} = \bigcap_{s \in S} \cls \bigcup_{t \in S, t \ge s} X_t
\end{equation*}
is called the \textit{limit set} of $(X_s)_{s \in S}$.
\end{definition}

\begin{remark}
When $S = \R^+$ or $\Z^+$ (the set of nonnegative real numbers or the set of nonnegative integers) and $X_s = \varPhi(\{s\} \times E)$ for some semiflow $\varPhi \colon S \times X \to X$ and for some subset $E \subset X$, the limit set $L(X_s)_{s \in S}$ is called the \textit{omega limit set} of $E$ for $\varPhi$.
It will be denoted by $\omega_\varPhi(E)$.
\end{remark}

The above limit set should be distinguished from the set-theoretic limit set.
We note that from Remark~\ref{rmk:set of cluster pts}, the limit set is a generalization of the set of cluster points for nets of points.
The following theorem gives characterizations of points belonging to a limit set.
It is considered to be a generalization of Theorem~\ref{thm:cluster pt and convergent subnet}.

\begin{theorem}\label{thm:characterization of limit set}
Let $S = (S, \le)$ be a directed set, $(X_s)_{s \in S}$ be a net of subsets of $X$, and $y \in X$ be given.
Then the following properties are equivalent:
\begin{enumerate}[label=\textup{(\alph*)}]
\item $y \in L(X_s)_{s \in S}$.
\item For every directed set $I$ and every subnet $(X_{s_i})_{i \in I}$ of $(X_s)_{s \in S}$, there exist a directed set $J$, a monotone final map $h \colon J \to I$, and $(y_j)_{j \in J} \in \prod_{j \in J} \bigcup_{t \ge s_{h(j)}} X_t$ such that the net  $(y_j)_{j \in J}$ in $X$ converges to $y$.
\item There exist a directed set $I$, a monotone final map $h \colon I \to S$, and
\begin{equation*}
	(y_i)_{i \in I} \in \textstyle\prod_{i \in I} \bigcup_{t \in S, t \ge h(i)} X_t
\end{equation*}
such that the net $(y_i)_{i \in I}$ in $X$ converges to $y$.
\item There exist a directed set $I$, a subnet $(X_{s_i})_{i \in I}$ of $(X_s)_{s \in S}$, and $(y_i)_{i \in I} \in \prod_{i \in I} X_{s_i}$ such that the net $(y_i)_{i \in I}$ in $X$ converges to $y$.
\end{enumerate}
\end{theorem}

\begin{proof}
(a) $\Rightarrow$ (b):
Let $I = (I, \le)$ be a directed set and $(X_{s_i})_{i \in I}$ be a subnet of $(X_s)_{s \in S}$.
We consider the subset $J$ of the product directed set $I \times \mathcal{N}_y$ given by
\begin{equation*}
	J := \Set{(i, U) \in I \times \mathcal{N}_y}{\textstyle\bigcup_{t \ge s_i} X_t \cap U \ne \emptyset}.
\end{equation*}
By the definition of $L(X_s)_{s \in S}$, $J$ is directed.
We can choose a net $(y_{i, U})_{(i, U) \in J}$ so that
\begin{equation*}
	y_{i, U} \in \bigcup_{t \ge s_i} X_t \cap U
\end{equation*}
for all $(i, U) \in J$.
By defining a monotone final map $h \colon J \to I$ by
\begin{equation*}
	h(i, U) = i \mspace{20mu} ((i, U) \in J),
\end{equation*}
it holds that $y_{i, U} \in \bigcup_{t \ge s_{h(i, U)}} X_t$ for all $(i, U) \in J$ and $(y_{i, U})_{(i, U) \in J}$ converges to $y$.

(b) $\Rightarrow$ (c):
This is obvious.

(c) $\Rightarrow$ (d):
For each $i \in I$, there exists $s_i \in S$ such that $s_i \ge h(i)$ and $y_i \in X_{s_i}$.
Since $s_i \to \bd(S)$ as $i \to \bd(I)$, (d) holds.

(d) $\Rightarrow$ (a):
Let $s \in S$ be fixed.
Then there is $i_0 \in I$ such that for all $i \in I$, $i \ge i_0$ implies $s_i \ge s$.
Let $I_0 := \Set{i \in I}{i \ge i_0}$ be a directed subset of $I$.
Since $(y_i)_{i \in I_0}$ is a net in $\bigcup_{t \in S, t \ge s} X_t$ converging to $y$, we have
\begin{equation*}
	y \in \cls \bigcup_{t \in S, t \ge s} X_t
\end{equation*}
from Lemma~\ref{lem:closure and net}.
This shows $y \in L(X_s)_{s \in S}$.
\end{proof}

We now compare the limit set introduced in Definition~\ref{dfn:limit set} with the upper and lower limit of sequences of subsets of (pseudo-) metric spaces.

\begin{notation}
Let $d$ be a pseudo-metric on $X$.
For each $x \in X$ and each subset $A \subset X$, let
\begin{equation*}
	d(x, A) := \inf_{y \in A} d(x, y).
\end{equation*}
We interpret that $d(x, \emptyset)$ is equal to $\infty$.
Then $d(x, A) < \infty$ if and only if $A$ is nonempty.
For each nonempty subsets $A, B \subset X$, let
\begin{equation*}
	d(A; B) := \sup_{x \in A} \inf_{y \in B} d(x, y) = \sup_{x \in A} d(x, B)
\end{equation*}
We promise that $d(A; \emptyset) = \infty$ for any nonempty subset $A$ and $d(\emptyset; B) = 0$ for any nonempty subset $B$.
We note that $d(\emptyset; \emptyset)$ is not defined.
\end{notation}

\begin{remark}
Let $X = (X, d)$ be a pseudo-metric space.
For each subset $A \subset X$, the function $d(\argdot, A) \colon X \to \R$ satisfies the following properties:
\begin{itemize}
\item $d(x, A) = 0$ if and only if $x \in \cls(A)$.
\item For all $x, y \in X$, $|d(x, A) - d(y, A)| \le d(x, y)$.
\end{itemize}
In particular, the function $d(\argdot, A) \colon X \to \R$ is continuous.
\end{remark}

\begin{definition}[ref.\ \cite{Aubin--Frankowska 2009}]
Let $X = (X, d)$ be a pseudo-metric space and $(X_n)_{n = 1}^\infty$ be a sequence of subsets of $X$.
The \textit{upper limit} and the \textit{lower limit} are defined by
\begin{align*}
	\Limsup_{n \to \infty} X_n &= \Set{y \in X}{\liminf_{n \to \infty} d(y, X_n) = 0}, \\
	\Liminf_{n \to \infty} X_n &= \Set{y \in X}{\lim_{n \to \infty} d(y, X_n) = 0},
\end{align*}
respectively.
Here $\liminf_{n \to \infty} d(y, X_n) := \sup_{m \ge 1} \inf_{n \ge m} d(y, X_n)$.
\end{definition}

In this paper, a \textit{subsequence} of some sequence $(x_n)_{n = 1}^\infty$ means a subnet $(x_{n_k})_{k = 1}^\infty$ (see \cite{Kelley 1955}).
The following are characterizations of the upper limit and the lower limit of a given sequence of subsets of pseudo-metric spaces.

\begin{lemma}[cf.\ \cite{Aubin--Frankowska 2009}]\label{lem:upper limit}
Let $X = (X, d)$ be a pseudo-metric space, $(X_n)_{n = 1}^\infty$ be a sequence of subsets of $X$, and $y \in X$ be given.
Then the following properties are euivalent:
\begin{enumerate}[label=\textup{(\alph*)}]
\item $y \in \Limsup_{n \to \infty} X_n$.
\item There exist a subsequence $(X_{n_k})_{k = 1}^\infty$ of $(X_n)_{n = 1}^\infty$ and $(y_k)_{k = 1}^\infty \in \prod_{k = 1}^\infty X_{n_k}$ such that the sequence $(y_k)_{k = 1}^\infty$ in $X$ converges to $y$.
\end{enumerate}
\end{lemma}

\begin{proof}
(a) $\Rightarrow$ (b):
By definition,
\begin{equation*}
	\inf_{n \ge k} d(y, X_n) = 0
\end{equation*}
holds for all $k \ge 1$.
Therefore, for each given $k \ge 1$, there are an integer $n_k \ge k$ and $y_k \in X_{n_k}$ such that
\begin{equation*}
	d(y, y_k) < \frac{1}{k}.
\end{equation*}
This shows that (b) holds.

(b) $\Rightarrow$ (a):
Let $\ep > 0$ be given.
Then there is an integer $k_0 \ge 1$ such that for all $k \ge k_0$, $d(y, y_k) < \ep$ holds.
Let $m \ge 1$ be a fixed integer.
We can choose an integer $k \ge k_0$ so that $n_k \ge m$.
Since
\begin{equation*}
	\inf_{n \ge m} d(y, X_n) \le d(y, X_{n_k}) \le d(y, y_k) < \ep,
\end{equation*}
we have
\begin{equation*}
	\liminf_{n \to \infty} d(y, X_n) = \sup_{m \ge 1} \inf_{n \ge m} d(y, X_n) \le \ep.
\end{equation*}
This holds for arbitrary $\ep > 0$, and therefore, we have $\liminf_{n \to \infty} d(y, X_n) = 0$.
\end{proof}

\begin{lemma}[cf.\ \cite{Aubin--Frankowska 2009}]\label{lem:lower limit}
Let $X = (X, d)$ be a pseudo-metric space, $(X_n)_{n = 1}^\infty$ be a sequence of subsets of $X$, and $y \in X$ be given.
Then the following properties are equivalent:
\begin{enumerate}[label=\textup{(\alph*)}]
\item $y \in \Liminf_{n \to \infty} X_n$.
\item There exist an integer $n_0 \ge 1$ and $(x_n)_{n = n_0}^\infty \in \prod_{n = n_0}^\infty X_n$ such that the sequence $(x_n)_{n = n_0}^\infty$ in $X$ converges to $y$.
\end{enumerate}
\end{lemma}

\begin{proof}
(a) $\Rightarrow$ (b):
Since $d(y, X_n) \to 0$ as $n \to \infty$, there is $n_0 \ge 1$ such that $d(y, X_n) < \infty$ for all $n \ge n_0$.
We can choose a sequence $(x_n)_{n = n_0}^\infty$ in $X$ so that
\begin{equation*}
	x_n \in X_n \amd d(y, x_n) < d(y, X_n) + \frac{1}{n}
\end{equation*}
hold for all $n \ge n_0$.
Then the sequence $(x_n)_{n = n_0}^\infty$ converges to $y$ because $d(y, X_n) \to 0$ as $n \to \infty$.

(b) $\Rightarrow$ (a):
This is obvious because $d(y, X_n) \le d(y, x_n)$ holds for all $n \ge n_0$.
\end{proof}

\begin{remark}
In Lemma~\ref{lem:lower limit}, we can choose $n_0 = 1$ when each $X_n$ is nonempty.
In \cite{Aubin--Frankowska 2009}, this non-emptiness is implicitly assumed, where it is stated that $y \in \Limsup_{n \to \infty} X_n$ if and only if $y$ is a cluster point of some sequence $(x_n)_{n = 1}^\infty$ belonging to $\prod_{n = 1}^\infty X_n$.
\end{remark}

By combining Theorem~\ref{thm:characterization of limit set} and Lemma~\ref{lem:upper limit}, the following statement is obtained as a corollary.
The proof can be omitted.
We note that it is also mentioned in \cite{Aubin--Frankowska 2009}.

\begin{corollary}
Let $X = (X, d)$ be a pseudo-metric space and $(X_n)_{n = 1}^\infty$ be a sequence of subsets of $X$.
Then $L(X_n)_{n = 1}^\infty = \Limsup_{n \to \infty} X_n$ holds.
\end{corollary}

\subsection{Convergence from above and from below}

To introduce convergence concepts for nets of subsets of $X$, we recall the definitions of upper and lower semicontinuity of set-valued maps.

\begin{definition}[ref.\ \cite{Aubin--Cellina 1984}]\label{dfn:set-valued map}
Let $A$ and $B$ be sets.
A map $F \colon A \to 2^B$ is called a \textit{set-valued map} from $A$ to $B$.
Here $2^B$ denotes the set of all subsets of $B$.
The set-valued map $F$ is also denoted by $F \colon A \tto B$.
For each subset $A_0 \subset A$, let $F(A_0) := \bigcup_{a \in A_0} F(a)$.
\end{definition}

\subsubsection{Upper semicontinuity and convergence from above}

\begin{definition}[ref.\ \cite{Aubin--Cellina 1984}]\label{dfn:upper semicontinuity}
Let $Y$ be a topological space, $F \colon X \tto Y$ be a set-valued map, and $x \in X$ be given.
$F$ is said to be \textit{upper semicontinuous} at $x$ if for every neighborhood $U$ of $F(x)$, there exists $V \in \mathcal{N}_x$ such that $F(V) \subset U$ holds.
\end{definition}

\begin{remark}
When $x$ is an isolated point of $X$ (i.e., $\{x\}$ is an open set of $X$), the set-valued map $F$ is always upper semicontinuous at $x$.
\end{remark}

Based on this definition, we introduce the following.

\begin{definition}\label{dfn:convergence from above of a net of subsets}
Let $S = (S, \le)$ be a directed set, $(X_s)_{s \in S}$ be a net of subsets of $X$, and $A \subset X$ be a subset.
We say that $(X_s)_{s \in S}$ \textit{converges from above} to $A$ if for every neighborhood $U$ of $A$, there exists $s \in S$ such that $\bigcup_{t \in S, t \ge s} X_t \subset U$ holds.
\end{definition}

\begin{remark}
Let $Y$ be a topological space, $F \colon X \tto Y$ be a set-valued map, and $x \in X$ be given.
For every $V_1, V_2 \in \mathcal{N}_x$, $V_1 \le V_2$ implies $F(V_1) \supset F(V_2)$.
Therefore, $F$ is upper semicontinuous at $x$ if and only if the net $(F(V))_{V \in \mathcal{N}_x}$ of subsets converges from above to $F(x)$.
\end{remark}

In the rest of this subsubsection, we study the property of the convergence from above in pseudo-metric spaces.

\begin{notation}
Let $d$ be a pseudo-metric on $X$.
For each $x \in X$, each subset $A \subset X$, and each $r > 0$, let
\begin{equation*}
	B_d(x; r) := \Set{y \in X}{d(x, y) < r}, \mspace{15mu} B_d(A; r) := \bigcup_{x \in A} B_d(x; r).
\end{equation*}
\end{notation}

\begin{remark}
Let $X = (X, d)$ be a pseudo-metric space, $S = (S, \le)$ be a directed set, and $(X_s)_{s \in S}$ be a net of subsets of $X$.
Then the following properties are equivalent:
\begin{itemize}
\item For every $\ep > 0$, there exists $s_0 \in S$ such that for all $s \ge s_0$, $X_s \subset B_d(A; \ep)$ holds.
\item $d(X_s; A) \to 0$ as $s \to \bd(S)$.
\end{itemize}
\end{remark}

The following is a key lemma to study the property of the convergence from above in pseudo-metric spaces.
We give a proof for the sake of completeness although its statement is mentioned in \cite[page 45 and page 66]{Aubin--Cellina 1984}.

\begin{lemma}\label{lem:nbd of cpt in pseudo-metric sp}
Suppose that $X = (X, d)$ is a pseudo-metric space.
Let $K \subset X$ be a nonempty compact set.
Then for every neighborhood $U$ of $K$, there exists $\delta > 0$ such that $B_d(K; \delta) \subset U$ holds.
\end{lemma}

\begin{proof}
We only have to consider the case that $U$ is open.
Let $r \colon X \to \R$ be the continuous function defined by
\begin{equation*}
	r(x) = d(x, X \setminus U) \mspace{20mu} (x \in X).
\end{equation*}
Then for every $x \in X$, $d(x, y) < r(x)$ implies $y \in U$.
Furthermore, $r(x) > 0$ for all $x \in K$ since $X \setminus U$ is closed.
The extreme value theorem ensures the existence of $x_0 \in K$ satisfying
\begin{equation*}
	r_0 := r(x_0) = \inf_{x \in K} r(x).
\end{equation*}
This implies $r_0 > 0$, and we have
\begin{equation*}
	B_d(K; r_0) \subset \bigcup_{x \in K} B_d(x; r(x)) \subset U.
\end{equation*}
This shows the conclusion.
\end{proof}

We can obtain the following corollary by using Lemma~\ref{lem:nbd of cpt in pseudo-metric sp}.

\begin{corollary}\label{cor:convergence from above in pseudo-metric sp}
Suppose that $X = (X, d)$ is a pseudo-metric space.
Let $S = (S, \le)$ be a directed set, $(X_s)_{s \in S}$ be a net of subsets of $X$, and $K \subset X$ be a nonempty compact set.
Then the following properties are equivalent:
\begin{enumerate}[label=\textup{(\alph*)}]
\item $(X_s)_{s \in S}$ converges from above to $K$.
\item $\lim_{s \to \bd(S)} d(X_s; K) = 0$, i.e., the net $(d(X_s; K))_{s \in S}$ of nonnegative real numbers converges to $0$.
\end{enumerate}
\end{corollary}

When the above property (b) holds, we will say that $(X_s)_{s \in S}$ \textit{converges from above to $A$ with respect to $d$}.

\begin{remark}[ref.\ \cite{Hale 1988}]
Let $X = (X, d)$ be a metric space, $S = \R^+$ or $\Z^+$, and $\varPhi \colon S \times X \to X$ be a semiflow.
A nonempty set $A \subset X$ is said to \textit{attract} a subset $E \subset X$ under $\varPhi$ if the net $(\varPhi(\{s\} \times E))_{s \in S}$ converges from above to $A$ with respect to $d$.
\end{remark}

\subsubsection{Lower semicontinuity and convergence from below}

\begin{definition}[cf.\ \cite{Aubin--Cellina 1984, Aubin--Frankowska 2009}]\label{dfn:lower semicontinuity}
Let $Y$ be a topological space, $F \colon X \tto Y$ be a set-valued map, and $x \in X$ be given.
$F$ is said to be \textit{lower semicontinuous} at $x$ if for every $y \in F(x)$ and every $U \in \mathcal{N}_y$, there exists $V \in \mathcal{N}_x$ such that for all $x' \in V$, $F(x') \cap U \ne \emptyset$ holds.
We interpret that $F$ is lower semicontinuous at $x$ when $F(x)$ is empty.
\end{definition}

\begin{remark}
In \cite[Definition 2 in Chapter 1]{Aubin--Cellina 1984} or \cite[Definition 1.4.2]{Aubin--Frankowska 2009}, it is assumed that $F(x)$ is always nonempty for every $x \in X$, or $F$ is restricted to the \textit{domain of definition} defined by
\begin{equation*}
	\Set{x \in X}{F(x) \ne \emptyset}.
\end{equation*}
\end{remark}

\begin{remark}
In the same way as the upper semicontinuity, the set-valued map $F$ is always lower semicontinuous at $x$ when $x$ is an isolated point of $X$.
\end{remark}

Based on this definition, we introduce the following convergence concept.

\begin{definition}\label{dfn:convergence from below of a net of subsets}
Let $S = (S, \le)$ be a directed set, $(X_s)_{s \in S}$ be a net of subsets of $X$, and $A \subset X$ be a given subset.
We say that $(X_s)_{s \in S}$ \textit{converges from below} to $A$ if for every $y \in A$ and every $U \in \mathcal{N}_y$, there exists $s_0 \in S$ such that for all $s \in S$, $s \ge s_0$ implies $X_s \cap U \ne \emptyset$.
We also say that $(X_s)_{s \in S}$ \textit{converges} to $A$ if $(X_s)_{s \in S}$ converges from above and from below to $A$.
\end{definition}

By definition, when $(X_s)_{s \in S}$ is single-valued, i.e., $X_s = \{x_s\}$ holds for all $s \in S$ for some net $(x_s)_{s \in S}$ in $X$, the convergence from below of $(X_s)_{s \in S}$ to some subset $A$ implies the convergence from above of $(X_s)_{s \in S}$ to $A$.

\begin{notation}
Suppose that $X = (X, d)$ is a pseudo-metric space.
Let $x \in X$ be given.
We consider the binary relation $\le$ on $X \setminus \{x\}$ defined as follows: $x_1 \le x_2$ if $d(x, x_1) \ge d(x, x_2)$.
Then $\le$ becomes a preorder.
Furthermore, if $x$ is not an isolated point of $X$, then $X \setminus \{x\} = (X \setminus \{x\}, \le)$ becomes a directed set.
\end{notation}

\begin{remark}
Suppose that $X = (X, d)$ is a pseudo-metric space.
Let $Y$ be a topological space, $F \colon X \tto Y$ be a set-valued map, and $x \in X$ be not an isolated point.
For any $x' \in X \setminus \{x\}$, $x' \ge x_0$ for some $x_0 \in X \setminus \{x\}$ is equivalent to the property that $x'$ belongs to the closed ball $\bar{B}_d(x; d(x; x_0)) := \Set{y \in X}{d(x, y) \le d(x, x_0)}$.
Therefore, $F$ is lower semicontinuous at $x$ if and only if the net $(F(x'))_{x' \in X \setminus \{x\}}$ of subsets converges from below to $F(x)$.
\end{remark}

The following is a generalization of the characterization of the continuity of maps between topological spaces in terms of nets of points.

\begin{theorem}[cf.\ \cite{Aubin--Cellina 1984}]\label{thm:characterization of lower semicontinuity}
Let $Y$ be a topological space, $F \colon X \tto Y$ be a set-valued map, and $x \in X$.
Then the following properties are equivalent:
\begin{enumerate}[label=\textup{(\alph*)}]
\item $F$ is lower semicontinuous at $x$.
\item For every directed set $A = (A, \le)$ and every net $(x_\alpha)_{\alpha \in A}$ in $X$ converging to $x$, the net $(F(x_\alpha))_{\alpha \in A}$ of subsets converges from below to $F(x)$.
\item For every $y \in F(x)$, every directed set $A = (A, \le)$, and every net $(x_\alpha)_{\alpha \in A}$ in $X$ converging to $x$, there exists a subnet $(x_{\alpha_\beta})_{\beta \in B}$ for some directed set $B = (B, \le)$ such that $(y_\beta)_{\beta \in B} \in \prod_{\beta \in B} F(x_{\alpha_\beta})$ and the net $(y_\beta)_{\beta \in B}$ in $X$ converges to $y$.
\end{enumerate}
\end{theorem}

\begin{proof}
We only have to consider the case that $F(x)$ is nonempty.

(a) $\Rightarrow$ (b):
Let $A = (A, \le)$ be a directed set and $(x_\alpha)_{\alpha \in A}$ be a net in $X$ converging to $x$.
Let $y \in F(x)$ and $U \in \mathcal{N}_y$ be given.
Since $F$ is lower semicontinuous at $x$, there is $V \in \mathcal{N}_x$ such that for all $x' \in V$, $F(x') \cap U \ne \emptyset$ holds.
Then there is $\alpha_0 \in A$ such that for all $\alpha \in A$, $\alpha \ge \alpha_0$ implies $x_\alpha \in V$, which implies $F(x_\alpha) \cap U \ne \emptyset$ for such $\alpha$.
Therefore, (b) holds.

(b) $\Rightarrow$ (c):
Let $y \in F(x)$, $A = (A, \le)$ be a directed set, and $(x_\alpha)_{\alpha \in A}$ be a net in $X$ converging to $x$.
Let $B$ be the subset of the product directed set $A \times \mathcal{N}_y$ given by
\begin{equation*}
	B := \Set{(\alpha, U) \in A \times \mathcal{N}_y}{F(x_\alpha) \cap U \ne \emptyset}.
\end{equation*}
By the assumption, $B$ is nonempty and directed.
Then we can choose a net $(y_{\alpha, U})_{(\alpha, U) \in B}$ in $X$ so that
\begin{equation*}
	y_{\alpha, U} \in F(x_\alpha) \cap U
\end{equation*}
for all $(\alpha, U) \in B$.
By this choice, the net $(y_{\alpha, U})_{(\alpha, U) \in B}$ converges to $y$.
We consider the monotone final map $h \colon B \to A$ defined by
\begin{equation*}
	h(\alpha, U) = \alpha \mspace{20mu} ((\alpha, U) \in B).
\end{equation*}
Then we have $y_{\alpha, U} \in F(x_{h(\alpha, U)})$ for every $(\alpha, U) \in B$.
Therefore, (c) holds.

(c) $\Rightarrow$ (a):
We suppose the contrary and derive a contradiction.
Then there are $y \in F(x)$ and an open neighborhood $U$ of $x$ with the following property: For every $V \in \mathcal{N}_x$, there is $x_V \in V$ such that $F(x_V) \cap U = \emptyset$.
Since the net $(x_V)_{V \in \mathcal{N}_x}$ converges to $x$, there is a subnet $(x_{V_\beta})_{\beta \in B}$ for some directed set $B$ such that $(y_\beta)_{\beta \in B} \in \prod_{\beta \in B} F(x_{V_\beta})$ and the net $(y_\beta)_{\beta \in B}$ converges to $y$.
This $(y_\beta)_{\beta \in B}$ is a net in $X \setminus U$ because
\begin{equation*}
	F(x_V) \cap U = \emptyset
\end{equation*}
for all $V \in \mathcal{N}_x$.
Therefore, we have $y \in X \setminus U$ from Lemma~\ref{lem:closure and net}, which is a contradiction.
Thus, (a) holds.
\end{proof}

\begin{remark}
In \cite{Aubin--Cellina 1984}, it is stated that $F$ is lower semicontinuous at $x$ if and only if for any $y \in F(x)$, any directed set $A$, and any net $(x_\alpha)_{\alpha \in A}$ in $X$ converging to $x$, there exists $(y_\alpha)_{\alpha \in A} \in \prod_{\alpha \in A} F(x_\alpha)$ such that the net $(y_\alpha)_{\alpha \in A}$ converges to $y$.
\end{remark}

The following theorem gives a characterization of the property of the convergence from below.
The proof is similar to that of Theorem~\ref{thm:characterization of lower semicontinuity}, but there is a slight difference.
We give a proof for the sake of completeness.

\begin{theorem}\label{thm:characterization of convergence from below}
Let $S = (S, \le)$ be a directed set, $(X_s)_{s \in S}$ be a net of subsets of $X$, and $A \subset X$ be a subset.
Then the following properties are equivalent:
\begin{enumerate}[label=\textup{(\alph*)}]
\item $(X_s)_{s \in S}$ converges from below to $A$.
\item The following statement holds for each $y \in A$: For every directed set $I$ and every subnet $(X_{s_i})_{i \in I}$ of $(X_s)_{s \in S}$, there exist a directed set $J$, monotone final map $h \colon J \to I$, and $(y_j)_{j \in J} \in \prod_{j \in J} X_{s_{h(j)}}$ such that the net $(y_j)_{j \in J}$ in $X$ converges to $y$.
\item The following statement holds for each $y \in A$: For every directed set $I$ and every subnet $(X_{s_i})_{i \in I}$ of $(X_s)_{s \in S}$, there exist a directed set $J$, a subnet $(s_{i_j})_{j \in J}$ of $(s_i)_{i \in I}$, and $(y_j)_{j \in J} \in \prod_{j \in J} X_{s_{i_j}}$ such that the net $(y_j)_{j \in J}$ in $X$ converges to $y$.
\end{enumerate}
\end{theorem}

\begin{proof}
We only have to consider the case that $A$ is nonempty.

(a) $\Rightarrow$ (b):
Let $y \in A$, $I = (I, \le)$ be a directed set, and $(X_{s_i})_{i \in I}$ be a subnet of $(X_s)_{s \in S}$.
We consider the subset $J$ of the product directed set $I \times \mathcal{N}_y$ given by
\begin{equation*}
	J := \Set{(i, U) \in I \times \mathcal{N}_y}{X_{s_i} \cap U \ne \emptyset}.
\end{equation*}

\textbf{Step 1.}
We claim that $J$ is a directed set.
We first show that $J$ is nonempty.
We choose some $U \in \mathcal{N}_y$.
By the assumption, there is $t_U \in S$ such that for all $s \in S$, $s \ge t_U$ implies $X_s \cap U \ne \emptyset$.
Since $s_i \to \bd(S)$ as $i \to \bd(I)$, there is $i_U \in I$ such that for all $i \in I$, $i \ge i_U$ implies $s_i \ge t_U$.
This means
\begin{equation*}
	(i, U) \in J
\end{equation*}
for all $i \ge i_U$, and therefore, $J$ is nonempty.
We next show that $J = (J, \le)$ is directed.
Let $(i_1, U_1), (i_2, U_2) \in J$ be given.
Let $U := U_1 \cap U_2 \in \mathcal{N}_y$.
By choosing an upper bound $i \in I$ of $i_1, i_2, i_U$, we obtain an upper bound $(i, U) \in J$ of $(i_1, U_1), (i_2, U_2)$.
Therefore, $J$ is directed.

\textbf{Step 2.}
We can choose a net $(y_{i, U})_{(i, U) \in J}$ so that
\begin{equation*}
	y_{i, U} \in X_{s_i} \cap U
\end{equation*}
for all $(i, U) \in J$.
By defining a monotone final map $h \colon J \to I$ by
\begin{equation*}
	h(i, U) = i \mspace{20mu} ((i, U) \in J),
\end{equation*}
it holds that $y_{i, U} \in X_{s_{h(i, U)}}$ for all $(i, U) \in J$ and $(y_{i, U})_{(i, U) \in J}$ converges to $y$.
Therefore, (b) holds.

(b) $\Rightarrow$ (c):
This is obvious.

(c) $\Rightarrow$ (a):
We suppose the contrary and derive a contradiction.
Then there are $y \in A$ and an open neighborhood $U$ of $y$ with the following property: For every $i \in S$, there is $s_i \in S$ such that
\begin{equation*}
	s_i \ge i \amd X_{s_i} \cap U = \emptyset.
\end{equation*}
Since $s_i \to \bd(S)$ as $i \to \bd(S)$, $(X_{s_i})_{i \in S}$ is a subnet of $(X_s)_{s \in S}$.
Then we can choose a directed set $J$, a subnet $(s_{i_j})_{j \in J}$ of $(s_i)$, and a net $(y_j)_{j \in J}$ so that $y_j \in X_{s_{i_j}}$ for all $j \in J$ and $(y_j)_{j \in J}$ converges to $y$.
This $(y_j)_{j \in J}$ is a net in $X \setminus U$ because
\begin{equation*}
	X_{s_i} \subset X \setminus U
\end{equation*}
holds for all $i \in S$.
Therefore, we have $y \in X \setminus U$ from Lemma~\ref{lem:closure and net}, which is a contradiction.
Thus, (a) holds.
\end{proof}

By combining Theorems~\ref{thm:characterization of limit set} and \ref{thm:characterization of convergence from below}, we obtain the following corollary.
The proof can be omitted.

\begin{corollary}\label{cor:convergence from below and limit set}
Let $S$ be a directed set, $(X_s)_{s \in S}$ be a net of subsets of $X$, and $A \subset X$ be a given subset.
If $(X_s)_{s \in S}$ converges from below to $A$, then $A \subset L(X_s)_{s \in S}$ holds.
\end{corollary}

In the rest of this subsubsection, we study the property of the convergence from below in pseudo-metric spaces.

\begin{remark}
Let $X = (X, d)$ be a pseudo-metric space, $S = (S, \le)$ be a directed set, $(X_s)_{s \in S}$ be a net of subsets of $X$, and $y \in X$ be given.
Then the following properties are equivalent:
\begin{itemize}
\item For every $\ep > 0$, there exists $s_0 \in S$ such that for all $s \ge s_0$, $X_s \cap B_d(y; \ep)$ holds.
\item $\lim_{s \to \bd(S)} d(y, X_s) = 0$.
\end{itemize}
\end{remark}

In view of the above remark, $d(A; X_s) \to 0$ as $s \to \bd(S)$ is sufficient for the convergence from below of $(X_s)_{s \in S}$ to $A$.
As the following theorem shows, the condition is also necessary when $A$ is compact.

\begin{theorem}
Suppose that $X = (X, d)$ is a pseudo-metric space.
Let $S = (S, \le)$ be a directed set, $(X_s)_{s \in S}$ be a net of subsets of $X$, and $K \subset X$ be a nonempty compact set.
Then $(X_s)_{s \in S}$ converges from below to $K$ if and only if $\lim_{s \to \bd(S)} d(K; X_s) = 0$.
\end{theorem}

The proof is similar to that of the following corresponding theorem about the lower semicontinuity of set-valued maps.
Therefore, it can be omitted.

\begin{theorem}[ref.\ \cite{Aubin--Cellina 1984}]
Let $Y = (Y, \rho)$ be a pseudo-metric space, $F \colon X \tto Y$ be a set-valued map, and $x \in X$ be given.
Suppose that $F(x)$ is compact.
Then $F$ is lower semicontinuous at $x$ if and only if $\lim_{x' \to x} \rho(F(x); F(x')) = 0$ holds.
\end{theorem}

By $\lim_{x' \to x} \rho(F(x); F(x')) = 0$, we mean that for every $\ep > 0$, there exists $V \in \mathcal{N}_x$ such that for all $x' \in V$, $\rho(F(x); F(x')) < \ep$ holds.
See \cite[Proof of Proposition 3 in Chapter 1]{Aubin--Cellina 1984} for the proof.

\subsection{Convergence and separation axioms}

In this subsection, we will investigate the connection between the convergence property and limit sets under the separation axioms.

\subsubsection{Separation axioms and compactness}

For the purpose stated above, we recall the separation axioms for topological spaces.

\begin{definition}[ref.\ \cite{Kelley 1955}]\label{dfn:Hausdorff sp}
$X$ is said to be \textit{Hausdorff} if for every $x, y \in X$ satisfying $x \ne y$, there exist $U \in \mathcal{N}_x$ and $V \in \mathcal{N}_y$ such that $U \cap V = \emptyset$.
\end{definition}

\begin{theorem}[ref.\ \cite{Kelley 1955}]\label{thm:separation of cpt and pt in Hausdorff sp}
Suppose that $X$ is Hausdorff and let $K \subset X$ be a compact set.
Then for every $y \in X \setminus K$, there are neighborhood $U$ of $K$ and a neighborhood $V$ of $y$ such that $U \cap V = \emptyset$.
\end{theorem}

Here a topological space is said to be \textit{compact} if every open cover has a finite subcover (ref.\ \cite{Kelley 1955}).
The compactness can be paraphrased by using the finite intersection property.

\begin{definition}[ref.\ \cite{Kelley 1955}]
Let $\mathcal{C}$ be a collection of subsets of $X$.
$\mathcal{C}$ is said to have the \textit{finite intersection property} if $\bigcap \mathcal{C}_0 := \bigcap_{C \in \mathcal{C}_0} C$ is nonempty for every finite sub-collection $\mathcal{C}_0 \subset \mathcal{C}$.
\end{definition}

\begin{lemma}[ref.\ \cite{Kelley 1955}]\label{lem:cptness and finite intersection property}
$X$ is compact if and only if for every collection of closed sets of $X$ with the finite intersection property, the intersection is nonempty.
\end{lemma}

Then the compactness can be characterized by the existence of a convergent subnet.

\begin{theorem}[ref.\ \cite{Kelley 1955}]\label{thm:cptness and net}
$X$ is compact if and only if every net in $X$ has a cluster point.
Consequently, $X$ is compact if and only if every net in $X$ has a convergent subnet.
\end{theorem}

\begin{definition}[ref.\ \cite{Kelley 1955}]
$X$ is said to be \textit{regular} if for every $x \in X$ and every $U \in \mathcal{N}_x$, there exists a closed neighborhood $V$ of $x$ contained in $U$.
\end{definition}

\begin{remark}
$X$ is regular if and only if for every closed set $F$ of $X$ and every $x \not\in F$, there exist open sets $U$ and $V$ of $X$ such that $x \in U$, $F \subset V$, and $U \cap V \ne \emptyset$.
\end{remark}

\subsubsection{Convergence in Hausdorff or regular spaces}

\begin{lemma}[cf.\ \cite{Chepyzhov--Vishik 2002}]\label{lem:convergence to cpt in Hausdorff sp}
Suppose that $X$ is Hausdorff.
Let $S = (S, \le)$ be a directed set and $(X_s)_{s \in S}$ be a net of subsets of $X$.
If $(X_s)_{s \in S}$ converges from above to some compact set $K \subset X$, then $L(X_s)_{s \in S} \subset K$ holds.
Consequently, if $(X_s)_{s \in S}$ converges to $K$, then $L(X_s)_{s \in S} = K$ holds.
\end{lemma}

\begin{proof}
We suppose $L(X_s)_{s \in S} \not\subset K$ and derive a contradiction.
Then we can choose $y \in X$ so that
\begin{equation*}
	y \in L(X_s)_{s \in S} \cap (X \setminus K).
\end{equation*}
From Theorem~\ref{thm:separation of cpt and pt in Hausdorff sp}, there are a neighborhood $U$ of $K$ and a neighborhood $V$ of $y$ such that $U \cap V = \emptyset$.
For this $U$, there is $s \in S$ such that $\bigcup_{t \in S, t \ge s} X_t \subset U$.
Therefore, we have
\begin{equation*}
	\bigcup_{t \in S, t \ge s} X_t \cap V \subset U \cap V = \emptyset,
\end{equation*}
which implies $y \not\in \cls \bigcup_{t \in S, t \ge s} X_t$.
This contradicts $y \in L(X_s)_{s \in S}$.
$L(X_s)_{s \in S} = K$ under the convergence of $(X_s)_{s \in S}$ to $K$ is a consequence of Corollary~\ref{cor:convergence from below and limit set}.
\end{proof}

\begin{remark}
Let $\varPhi \colon \R^+ \times X \to X$ be a semiflow and $E, A \subset X$ be subsets.
$A$ is said to \textit{attract} $E$ under $\varPhi$ if the net $(\varPhi(\{s\} \times E))_{s \in \R^+}$ of subsets of $X$ converges from above to $A$.
We note that this is not equivalent to the attraction in metric spaces.
In \cite[Proposition 2.3 in Chapter XI]{Chepyzhov--Vishik 2002}, it is proved that if $A$ is compact and attracts $E$ under $\varPhi$, then $\omega_\varPhi(E) \subset A$ holds.
\end{remark}

\begin{lemma}[cf.\ \cite{Marzocchi--ZandonellaNecca 2002}]\label{lem:convergence in regular sp}
Suppose that $X$ is regular.
Let $S = (S, \le)$ be a directed set and $(X_s)_{s \in S}$ be a net of subsets of $X$.
If $(X_s)_{s \in S}$ converges from above to some subset $A \subset X$, then $L(X_s)_{s \in S} \subset \cls(A)$ holds.
Consequently, if $(X_s)_{s \in S}$ converges to $A$, then $L(X_s)_{s \in S} = \cls(A)$ holds.
\end{lemma}

We omit the proof because the argument is similar to the proof of Lemma~\ref{lem:convergence to cpt in Hausdorff sp}.
See also \cite[Proof of Theorem 2.10]{Marzocchi--ZandonellaNecca 2002}.

\begin{remark}
Suppose that $X$ is a regular Hausdorff space.
Let $\varPhi \colon \R^+ \times X \to X$ be a semiflow and $E, A \subset X$ be subsets.
In \cite[Theorem 2.10]{Marzocchi--ZandonellaNecca 2002}, it is proved that if $A$ attracts $E$ under $\varPhi$, then $\omega_\varPhi(E) \subset \cls(A)$ holds.
\end{remark}

\section{Asymptotic compactness and limit set compactness}\label{sec:asymptotic cptness and limit set cptness}

Throughout this section, let $X$ be a topological space.
The purpose of this section is to introduce the notions of asymptotic compactness and weak asymptotic compactness and investigate their fundamental properties.
We also study their connection with the limit set compactness and the eventual Lagrange stability introduced below.

\subsection{Asymptotic compactness}

In view of Theorem~\ref{thm:characterization of limit set}, we introduce the following notions of asymptotic compactness for nets of nonempty subsets.
They are considered to be generalizations of the asymptotic compactness of subsets under semiflows with continuous time in metric spaces (see \cite{Sell--You 2002}).
See also Definition~\ref{dfn:asymptotic seq cptness}.

\begin{definition}\label{dfn:asymptotic cptness}
Let $S$ be a directed set and $(X_s)_{s \in S}$ be a net of subsets of $X$.
We say that $(X_s)_{s \in S}$ is \textit{asymptotically compact} if for every directed set $I$, every subnet $(X_{s_i})_{i \in I}$ of $(X_s)_{s \in S}$, and every $(y_i)_{i \in I} \in \prod_{i \in I} X_{s_i}$, the net $(y_i)_{i \in I}$ in $X$ has a convergent subnet.
\end{definition}

\begin{remark}
If there exists $s_0 \in S$ such that $X_s = \emptyset$ for all $s \ge s_0$, then it holds that $(X_s)_{s \in S}$ is asymptotically compact.
\end{remark}

\begin{definition}\label{dfn:weak asymptotic cptness}
Let $S$ be a directed set and $(X_s)_{s \in S}$ be a net of subsets of $X$.
We say that $(X_s)_{s \in S}$ is \textit{weakly asymptotically compact} if for every directed set $I$, every monotone final map $h \colon I \to S$, and every $(y_i)_{i \in I} \in \prod_{i \in I} \bigcup_{t \in S, t \ge h(i)} X_t$, the net $(y_i)_{i \in I}$ in $X$ has a convergent subnet.
\end{definition}

By definition, the asymptotic compactness implies the weak asymptotic compactness.
It is not apparent whether the converse holds or not.
We note that every limit $y$ of a convergent subnet of $(y_i)_{i \in I}$ in Definitions~\ref{dfn:asymptotic cptness} and \ref{dfn:weak asymptotic cptness} necessarily belongs to $L(X_s)_{s \in S}$ from Theorem~\ref{thm:characterization of limit set}.

\begin{remark}
Suppose that $(X_s)_{s \in S}$ is weakly asymptotically compact.
Then for every $(x_s)_{s \in S} \in \prod_{s \in S} X_s$, the net $(x_s)_{s \in S}$ in $X$ has a convergent subnet.
Therefore, $L(X_s)_{s \in S}$ is nonempty if every $X_s$ is nonempty.
\end{remark}

\begin{lemma}\label{lem:weak asymptotic cptness and limit set}
Let $S = (S, \le)$ be a directed set and $(X_s)_{s \in S}$ be a net of subsets of $X$.
If $(X_s)_{s \in S}$ is weakly asymptotically compact, then $(X_s)_{s \in S}$ converges from above to $L(X_s)_{s \in S}$.
\end{lemma}

\begin{proof}
We suppose the contrary and derive a contradiction.
Then there is an open neighborhood $U$ of $L(X_s)_{s \in S}$ such that for all $s \in S$, $\bigcup_{t \in S, t \ge s} X_t \cap (X \setminus U) \ne \emptyset$ holds.
Therefore, we can choose a net $(y_s)_{s \in S}$ in $X$ so that
\begin{equation*}
	y_s \in \bigcup_{t \in S, t \ge s} X_t \cap (X \setminus U)
\end{equation*}
for all $s \in S$.
By the weak asymptotic compactness, the net $(y_s)_{s \in S}$ has a convergent subnet.
From Lemma~\ref{lem:closure and net} and Theorem~\ref{thm:characterization of limit set}, we have $L(X_s)_{s \in S} \cap (X \setminus U) \ne \emptyset$.
This is a contradiction.
\end{proof}

\begin{remark}
The following properties are equivalent from Theorem~\ref{thm:cptness and net}:
\begin{enumerate}[label=(\alph*)]
\item $X$ is compact.
\item For every directed set $S = (S, \le)$ and every net $(X_s)_{s \in S}$ of nonempty subsets of $X$, $(X_s)_{s \in S}$ is asymptotically compact.
\item For every directed set $S = (S, \le)$ and every net $(X_s)_{s \in S}$ of nonempty subsets of $X$, $(X_s)_{s \in S}$ is weakly asymptotically compact.
\end{enumerate}
\end{remark}

We next study the asymptotic compactness when $X$ is locally compact.
Here $X$ is said to be \textit{locally compact} if every point in $X$ has a compact neighborhood.
It is straightforward to show that every compact set has a compact neighborhood when $X$ is locally compact.

\begin{lemma}\label{lem:convergence of net to cpt in locally cpt sp}
Suppose that $X$ is locally compact.
Then for every directed set $A = (A, \le)$ and every net $(x_\alpha)_{\alpha \in A}$ in $X$ converging from above to some nonempty compact set $K \subset X$, $(x_\alpha)_{\alpha \in A}$ has a convergent subnet.
\end{lemma}

\begin{proof}
We can choose a compact neighborhood $U$ of $K$.
The assumption implies that there is $\alpha_0 \in A$ such that for all $\alpha \in A$, $\alpha \ge \alpha_0$ implies $x_\alpha \in U$.
Let $A_0 := \Set{\alpha \in A}{\alpha \ge \alpha_0}$ be a directed subset of $A$.
Then $(x_\alpha)_{\alpha \in A_0}$ has a convergent subnet by the compactness of $U$ from Theorem~\ref{thm:cptness and net}.
This shows that $(x_\alpha)_{\alpha \in A}$ also has a convergent subnet.
\end{proof}

\begin{lemma}\label{lem:convergence to cpt and asymptotic cptness in locally cpt sp}
Suppose that $X$ is locally compact.
Let $S$ be a directed set and $(X_s)_{s \in S}$ be a net of subsets of $X$.
If $(X_s)_{s \in S}$ converges from above to some compact set $K \subset X$, then $(X_s)_{s \in S}$ is asymptotically compact.
\end{lemma}

\begin{proof}
Let $I$ be a directed set and $(X_{s_i})_{i \in I}$ be a subnet of $(X_s)_{s \in S}$.
We only have to consider the case that $K$ and $\prod_{i \in I} X_{s_i}$ are nonempty.
Let $(y_i)_{i \in I} \in \prod_{i \in I} X_{s_i}$.
Since $(X_s)_{s \in S}$ converges from above to $K$, the net $(y_i)_{i \in I}$ in $X$ converges from above to $K$.
Therefore, $(y_i)_{i \in I}$ has a convergent subnet from Lemma~\ref{lem:convergence of net to cpt in locally cpt sp}.
This shows that $(X_s)_{s \in S}$ is asymptotically compact.
\end{proof}

\subsection{Limit set compactness and eventual Lagrange stability}

In this paper, we use the following terminologies.

\begin{definition}\label{dfn:limit set cptness}
Let $S = (S, \le)$ be a directed set and $(X_s)_{s \in S}$ be a net of subsets of $X$.
We say that $(X_s)_{s \in S}$ is \textit{limit set compact} if (i) the limit set $L(X_s)_{s \in S}$ is a nonempty compact set and (ii) $(X_s)_{s \in S}$ converges from above to $L(X_s)_{s \in S}$.
\end{definition}

\begin{definition}
Let $S = (S, \le)$ be a directed set and $(X_s)_{s \in S}$ be a net of subsets of $X$.
We say that $(X_s)_{s \in S}$ is \textit{eventually Lagrange stable} if there exists $s_0 \in S$ such that $\bigcup_{t \in S, t \ge s_0} X_t$ is relatively compact, i.e., the closure is compact.
\end{definition}

\begin{remark}
The terminology of \textit{positive Lagrange stability} has been used for a flow $\varPhi \colon \R^+ \times X \to X$ in a topological space $X$ as follows (e.g., see \cite{Deysach--Sell 1965}):
The motion $\varPhi(t, x)$ is said to be \textit{positively Lagrange stable} if the positive orbit $\Set{\varPhi(t, x)}{t \in \R^+}$ is relatively compact.
\end{remark}

The following theorem is considered to be a generalization of Theorem~\ref{thm:cptness and net}.

\begin{corollary}\label{cor:cptness and limit set cptness}
The following properties are equivalent:
\begin{enumerate}[label=\textup{(\alph*)}]
\item $X$ is compact.
\item For every directed set $S = (S, \le)$ and every net $(X_s)_{s \in S}$ of nonempty subset of $X$, $(X_s)_{s \in S}$ is limit set compact.
\item For every directed set $S = (S, \le)$ and every net $(X_s)_{s \in S}$ of nonempty subset of $X$, $L(X_s)_{s \in S} \ne \emptyset$.
\end{enumerate}
\end{corollary}

\begin{proof}
(a) $\Rightarrow$ (b):
The non-emptiness of $L(X_s)_{s \in S}$ follows by Lemma~\ref{lem:cptness and finite intersection property} because the net $(Y_s)_{s \in S}$ of subsets of $X$ given by 
\begin{equation*}
	Y_s := \cls \bigcup_{t \in S, t \ge s} X_t \mspace{20mu} (s \in S)
\end{equation*}
is a family of closed sets of $X$ with the finite intersection property.
The compactness of $L(X_s)_{s \in S}$ also follows by the compactness of $X$.
The convergence from above of $(X_s)_{s \in S}$ to $L(X_s)_{s \in S}$ is a consequence of Lemma~\ref{lem:weak asymptotic cptness and limit set}.

(b) $\Rightarrow$ (c): This is obvious.

(c) $\Rightarrow$ (a): This holds from Theorem~\ref{thm:cptness and net}.
\end{proof}

\begin{lemma}\label{lem:eventual Lagrange stability}
Let $S = (S, \le)$ be a directed set and $(X_s)_{s \in S}$ be a net of nonempty subsets of $X$.
If $(X_s)_{s \in S}$ is eventually Lagrange stable, then the following statements hold:
\begin{enumerate}[label=\textup{\arabic*.}]
\item $(X_s)_{s \in S}$ converges to some nonempty closed compact set.
\item $(X_s)_{s \in S}$ is asymptotically compact.
Furthermore, $(X_s)_{s \in S}$ is limit set compact.
\end{enumerate}
\end{lemma}

\begin{proof}
By the eventual Lagrange stability, there is $s_0 \in S$ such that $K := \cls \bigcup_{t \in S, t \ge s_0} X_t$ is a nonempty closed compact set.

1. Since $X_s \subset K$ holds for all $s \ge s_0$, it holds that $(X_s)_{s \in S}$ converges to $K$.

2. Let $I = (I, \le)$ be a directed set and $(X_{s_i})_{i \in I}$ be a subnet of $(X_s)_{s \in S}$.
We only have to consider the case that $\prod_{i \in I} X_{s_i}$ is nonempty.
Let $(y_i)_{i \in I} \in \prod_{i \in I} X_{s_i}$.
For the above $s_0$, there is $i_0 \in I$ such that $s_i \ge s_0$ holds for all $i \ge i_0$.
Let $I_0 := \Set{i \in I}{i \ge i_0}$ be a directed subset of $I$.
Then $(y_i)_{i \in I_0}$ becomes a net in $K$.
Therefore, this net has a convergent subnet from Theorem~\ref{thm:cptness and net}.
This implies that $(y_i)_{i \in I}$ also has a convergent subnet.
Thus, $(X_s)_{s \in S}$ is asymptotically compact.
From Lemma~\ref{lem:weak asymptotic cptness and limit set}, $(X_s)_{s \in S}$ converges from above to the nonempty closed set $L(X_s)_{s \in S}$.
Furthermore, $L(X_s)_{s \in S}$ is compact because $L(X_s)_{s \in S} \subset K$.
Therefore, the limit set compactness follows.
\end{proof}

We finally study a connection between the eventual Lagrange stability and the limit set compactness in locally compact regular spaces.
The following fact is crucial.

\begin{lemma}[ref.\ \cite{Kelley 1955}]\label{lem:cptness of closure of cpt in regular sp}
Suppose that $X$ is a regular space.
Then for every compact set $A \subset X$, $\cls(A)$ is also compact.
\end{lemma}

\begin{lemma}\label{lem:eventual Lagrange stability and limit set cptness in locally cpt regular sp}
Suppose that $X$ is a locally compact regular space.
Let $S = (S, \le)$ be a directed set and $(X_s)_{s \in S}$ be a net of nonempty subsets of $X$.
Then the following properties are equivalent:
\begin{enumerate}[label=\textup{(\alph*)}]
\item $(X_s)_{s \in S}$ converges from above to some nonempty compact set.
\item $(X_s)_{s \in S}$ is eventually Lagrange stable.
\item $(X_s)_{s \in S}$ is asymptotically compact and limit set compact.
\end{enumerate}
\end{lemma}

\begin{proof}
(a) $\Rightarrow$ (b):
Let $K \subset X$ be a nonempty compact set to which $(X_s)_{s \in S}$ converges from above.
We can choose a compact neighborhood $U$ of $K$ and $s_0 \in S$ such that $\bigcup_{t \in S, t \ge s_0} X_t \subset U$.
This shows $\cls \bigcup_{t \in S, t \ge s_0} X_t \subset \cls(U)$, where $\cls(U)$ is also compact from Lemma~\ref{lem:cptness of closure of cpt in regular sp}.
Therefore, $(X_s)_{s \in S}$ is eventually Lagrange stable.

(b) $\Rightarrow$ (c):
This follows by Lemma~\ref{lem:eventual Lagrange stability}.

(c) $\Rightarrow$ (a):
This is obvious.
\end{proof}

The following is a consequence of Lemma~\ref{lem:eventual Lagrange stability and limit set cptness in locally cpt regular sp} because every locally compact Hausdorff space is regular (ref.\ \cite{Kelley 1955}).
The proof can be omitted.

\begin{corollary}\label{cor:asymptotic cptness and limit set cptness in locally cpt Hausdorff sp}
Suppose that $X$ is a locally compact Hausdorff space.
Let $S = (S, \le)$ be a directed set and $(X_s)_{s \in S}$ be a net of nonempty subsets of $X$.
Then the following properties are equivalent:
\begin{enumerate}[label=\textup{(\alph*)}]
\item $(X_s)_{s \in S}$ converges from above to some nonempty compact set.
\item $(X_s)_{s \in S}$ is eventually Lagrange stable.
\item $(X_s)_{s \in S}$ is asymptotically compact and limit set compact.
\end{enumerate}
\end{corollary}

\section{Asymptotic compactness and limit set compactness in uniformizable spaces}
\label{sec:asymptotic cptness and limit set cptness in uniformizable spaces}

Let $X$ be a set.
In this section, we will prove that the asymptotic compactness and the limit set compactness are equivalent for any net of nonempty subsets in uniformizable spaces.

\subsection{Uniformity and uniform spaces}

For the purpose stated above, we recall the definition of uniform spaces.

\begin{notation}[ref.\ \cite{Kelley 1955}]
Let $X$ be a set.
\begin{itemize}
\item Let $\Delta_X$ denote the diagonal set $\Set{(x, x) \in X \times X}{x \in X}$.
\item For every subset $U \subset X \times X$, let
\begin{equation*}
	U^{-1} := \Set{(x, y) \in X \times X}{(y, x) \in U}.
\end{equation*}
\item For every subsets $U, V \subset X \times X$, let
\begin{equation*}
	V \circ U := \Set{(x, y) \in X \times X}{\text{$(x, z) \in U$ and $(z, y) \in V$ for some $z \in X$}}.
\end{equation*}
\end{itemize}
\end{notation}

\begin{definition}[ref.\ \cite{Kelley 1955}]
Let $X$ be a set.
A nonempty collection $\mathcal{U}$ of subsets of $X \times X$ is called a \textit{uniformity} if the following properties are satisfied:
\begin{enumerate}[label=(U\arabic*)]
\item Every $U \in \mathcal{U}$ contains the diagonal $\Delta_X$.
\item For every $U \in \mathcal{U}$, the inverse $U^{-1}$ also belongs to $\mathcal{U}$.
\item For every $U \in \mathcal{U}$, there exists $V \in \mathcal{U}$ such that the composition $V \circ V$ is a subset of $U$.
\item For every $U, V \in \mathcal{U}$, $U \cap V \in \mathcal{U}$.
\item For every $U \in \mathcal{U}$ and every subset $V \subset X \times X$ containing $U$, $V \in \mathcal{U}$ holds.
\end{enumerate}
$U \in \mathcal{U}$ is said to be \textit{symmetric} if $U^{-1} = U$.
In (U3), one can assume that $V$ is symmetric by considering $V \cap V^{-1}$.
The pair $(X, \mathcal{U})$ is called a \textit{uniform space}.
\end{definition}

We note that a uniformity $\mathcal{U}$ on $X$ is considered as a directed poset with the partial order $\le$ defined as follows: $U_1 \le U_2$ if $U_1 \supset U_2$.

\begin{notation}
Let $X$ be a set.
For every subset $U \subset X \times X$, every $x \in X$, and every subset $E \subset X$, let
\begin{equation*}
	U[x] := \Set{y \in X}{(x, y) \in U}, \mspace{15mu} U[E] := \bigcup_{x \in E} U[x].
\end{equation*}
We note that $U[\emptyset]$ is interpreted as $\emptyset$.
\end{notation}

A topological space $X$ is said to be \textit{uniformizable} if the topology is the uniform topology $\mathcal{T}$ of some uniformity $\mathcal{U}$ defined as follows: $T \in \mathcal{T}$ if for every $x \in T$, there exists $U \in \mathcal{U}$ such that $U[x] \subset T$.

\subsection{Convergence from above and asymptotic compactness}

In this subsection, we prove that the convergence from above to some nonempty compact set implies the asymptotic compactness.

\begin{lemma}\label{lem:directed set, convergence to cpt}
Suppose that $X = (X, \mathcal{U})$ is a uniform space.
Let $S = (S, \le)$ be a directed set, $(x_s)_{s \in S}$ be a net in $X$, and $K \subset X$ be a nonempty subset.
Let $I = (I, \le)$ be the subset of the product directed set $S \times \mathcal{U}$ defined by
\begin{equation*}
	I = \Set{(s, U) \in S \times \mathcal{U}}{x_s \in U[K] }.
\end{equation*}
If $(x_s)_{s \in S}$ converges from above to $K$, then for every $(s_1, U_1), (s_2, U_2) \in S \times \mathcal{U}$, there exists an upper bound $(s, U) \in I$ of $(s_1, U_1), (s_2, U_2)$.
\end{lemma}

\begin{proof}
We note that $I$ is nonempty by the assumption.
Let $(s_1, U_1), (s_2, U_2) \in S \times \mathcal{U}$ be given.
We choose $U := U_1 \cap U_2 \in \mathcal{U}$.
Since $U[K]$ is a neighborhood of $K$, there is $s_0 \in S$ such that for all $s \in S$, $s \ge s_0$ implies $x_s \in U[K]$.
Therefore, by choosing an upper bound $s \in S$ of $s_0, s_1, s_2$, we have an upper bound $(s, U) \in I$ of $(s_1, U_1), (s_2, U_2)$.
\end{proof}

\begin{remark}
In particular, the following hold:
\begin{enumerate}[label=\arabic*.]
\item $I = (I, \le)$ is a directed set.
\item $I$ is a cofinal subset of $S \times \mathcal{U}$.
\end{enumerate}
\end{remark}

The following theorem is an extension of Theorem~\ref{thm:cptness and net}.

\begin{theorem}\label{thm:convergence to cpt and convergent subnet in uniformizable sp}
Suppose that $X$ is a uniformizable space.
Then for every directed set $S = (S, \le)$ and every net $(x_s)_{s \in S}$ converging from above to some nonempty compact set $K \subset X$, $(x_s)_{s \in S}$ has a subnet converging to some point in $K$.
\end{theorem}

\begin{proof}
\textbf{Step 1.}
We choose a uniformity $\mathcal{U}$ so that the topology of $X$ is the uniform topology of $\mathcal{U}$.
Let $I$ be the directed set given in Lemma~\ref{lem:directed set, convergence to cpt}.
Then for every $(s, U) \in I$, there is $y_{s, U} \in K$ such that $x_s \in U[y_{s, U}]$, i.e., 
\begin{equation*}
	(y_{s, U}, x_s) \in U.
\end{equation*}
From Theorems~\ref{thm:cptness and net} and \ref{thm:cluster pt and convergent subnet}, there are a directed set $J = (J, \le)$ and a monotone final map $h \colon J \to I$ such that $(y_{h(j)})_{j \in J}$ converges to some point $y \in K$.
Let
\begin{equation*}
	(s_j, U_j) := h(j)
\end{equation*}
for each $j \in J$.
Since $I$ is a cofinal subset of $S \times \mathcal{U}$, the maps
\begin{equation*}
	J \ni j \mapsto s_j \in S,
	\mspace{15mu}
	J \ni j \mapsto U_j \in \mathcal{U}
\end{equation*}
are also monotone and final.

\textbf{Step 2.}
We claim that the subnet $(x_{s_j})_{j \in J}$ of $(x_s)_{s \in S}$ converges to $y$.
Let $U \in \mathcal{U}$ be given.
We choose a symmetric $U' \in \mathcal{U}$ so that $U' \circ U' \subset U$.
Then there is $j_1 \in J$ such that $j \ge j_1$ implies $U_j \ge U'$.
Since $(y_{h(j)})_{j \in J}$ converges to $y$, there is $j_2 \in J$ such that $j \ge j_2$ implies
\begin{equation*}
	y_{h(j)} \in U'[y].
\end{equation*}
Let $j_0 \in J$ be an upper bound of $j_1, j_2$.
Then for all $j \ge j_0$, we have $(y_{h(j)}, x_{s_j}) \in U_j \subset U'$ and $(y, y_{h(j)}) \in U'$, which shows
\begin{equation*}
	(y, x_{s_j}) \in U,
\end{equation*}
i.e., $x_{s_j} \in U[y]$.
Therefore, $(x_{s_j})_{j \in J}$ converges to $y$.
\end{proof}

As a corollary, we can obtain the asymptotic compactness from the convergence from above to some compact set.

\begin{corollary}\label{cor:convergence to cpt and asymptotic cptness in uniformizable sp}
Suppose that $X$ is a uniformizable space.
Let $S$ be a directed set and $(X_s)_{s \in S}$ be a net of subsets of $X$.
If $(X_s)_{s \in S}$ converges from above to some compact set $K \subset X$, then $(X_s)_{s \in S}$ is asymptotically compact.
\end{corollary}

\begin{proof}
Let $I$ be a directed set and $(X_{s_i})_{i \in I}$ be a subnet of $(X_s)_{s \in S}$.
We only have to consider the case that $K$ and $\prod_{i \in I} X_{s_i}$ are nonempty.
Let $(y_i)_{i \in I} \in \prod_{i \in I} X_{s_i}$ be given.
The convergence of $(X_s)_{s \in S}$ from above to $K$ implies that the net $(y_i)_{i \in I}$ in $X$ converges from above to $K$.
Therefore, $(y_i)_{i \in I}$ has a convergent subnet by applying Theorem~\ref{thm:convergence to cpt and convergent subnet in uniformizable sp}.
This shows that $(X_s)_{s \in S}$ is asymptotically compact.
\end{proof}

\subsection{Asymptotic compactness and limit set compactness}

In the following theorem, we will prove that the weak asymptotic compactness induces the limit set compactness.

\begin{theorem}\label{thm:asymptotic cptness and limit set cptness in uniformizable sp}
Suppose that $X$ is a uniformizable space.
Let $S = (S, \le)$ be a directed set and $(X_s)_{s \in S}$ be a net of subsets of $X$.
If $(X_s)_{s \in S}$ is weakly asymptotically compact, then $(X_s)_{s \in S}$ is compact.
Consequently, furthermore, if every $X_s$ is nonempty, then $(X_s)_{s \in S}$ is limit set compact.
\end{theorem}

\begin{proof}
Let $A = (A, \le)$ be a directed set and $(z_\alpha)_{\alpha \in A}$ be a net in $L(X_s)_{s \in S}$.
We will show that $(z_\alpha)_{\alpha \in A}$ has a subnet converging to some point in $L(X_s)_{s \in S}$.
Then the compactness of $L(X_s)_{s \in S}$ follows by Theorem~\ref{thm:cptness and net}.

\textbf{Step 1.}
We choose a uniformity $\mathcal{U}$ so that the topology of $X$ is the uniform topology of $\mathcal{U}$.
Let 
\begin{equation*}
	B := S \times \mathcal{U} \times A
\end{equation*}
be the product directed set.
By the definition of $L(X_s)_{s \in S}$, $U[z_\alpha] \cap \bigcup_{t \in S, t \ge s} X_t \ne \emptyset$ holds for every $(s, U, \alpha) \in B$.
Therefore, we can choose a net $(y_{s, U, \alpha})_{(s, U, \alpha) \in B}$ in $X$ so that
	\begin{equation*}
		y_{s, U, \alpha} \in U[z_\alpha] \cap \bigcup_{t \in S, t \ge s} X_t
	\end{equation*}
for all $(s, U, \alpha) \in B$.
Let $h \colon B \to S$ be the monotone final map defined by
\begin{equation*}
	h(s, U, \alpha) = s \mspace{20mu} ((s, U, \alpha) \in B).
\end{equation*}
Since
\begin{equation*}
	y_{s, U, \alpha} \in \bigcup_{t \in S, t \ge h(s, U, \alpha)} X_t
\end{equation*}
for all $(s, U, \alpha) \in B$, the net $(y_{s, U, \alpha})_{(s, U, \alpha) \in B}$ has a subnet converging to some $y \in L(X_s)_{s \in S}$ by the weak asymptotic compactness of $(X_s)_{s \in S}$.
From Theorem~\ref{thm:cluster pt and convergent subnet}, there are a directed set $C = (C, \le)$ and a monotone final map $g \colon C \to B$ such that the net $(y_{g(\gamma)})_{\gamma \in C}$ converges to $y$.

\textbf{Step 2.}
Let
\begin{equation*}
	(s_\gamma, U_\gamma, \alpha_\gamma) := g(\gamma).
\end{equation*}
for every $\gamma \in C$.
Then the maps
\begin{equation*}
	C \ni \gamma \mapsto s_\gamma \in S,
	\mspace{15mu}
	C \ni \gamma \mapsto U_\gamma \in \mathcal{U},
	\mspace{15mu}
	C \ni \gamma \mapsto \alpha_\gamma \in A
\end{equation*}
are also monotone and final.
We claim that the subnet $(z_{\alpha_\gamma})_{\gamma \in C}$ of $(z_\alpha)_{\alpha \in A}$ converges to $y$.
Let $U \in \mathcal{U}$ be given.
We choose a symmetric $U' \in \mathcal{U}$ so that $U' \circ U' \subset U$.
Then there exists $\gamma_0 \in C$ such that for all $\gamma \ge \gamma_0$,
\begin{equation*}
	U_\gamma \ge U', \mspace{15mu} y_{g(\gamma)} \in U'[y],
\end{equation*}
which shows $(z_{\alpha_\gamma}, y_{g(\gamma)}) \in U_\gamma \subset U'$.
Therefore, for all $\gamma \ge \gamma_0$, we have
\begin{equation*}
	(y, z_{\alpha_\gamma}) \in U.
\end{equation*}
This shows that $(z_{\alpha_\gamma})_{\gamma \in C}$ converges to $y$.
Finally, the limit set compactness of $(X_s)_{s \in S}$ under the non-emptiness of every $X_s$ is a consequence of Lemma~\ref{lem:weak asymptotic cptness and limit set}.
\end{proof}

We obtain the equivalence of the asymptotic compactness and the limit set compactness in uniformizable spaces by combining the results of this section.

\begin{theorem}\label{thm:equivalence between asymptotic cptness and limit set cptness in uniformizable sp}
Suppose that $X$ is a uniformizable space.
Let $S$ be a directed set and $(X_s)_{s \in S}$ be a net of nonempty subsets of $X$.
Then the following properties are equivalent:
\begin{enumerate}[label=\textup{(\alph*)}]
\item $(X_s)_{s \in S}$ converges from above to some nonempty compact set.
\item $(X_s)_{s \in S}$ is asymptotically compact.
\item $(X_s)_{s \in S}$ is weakly asymptotically compact.
\item $(X_s)_{s \in S}$ is limit set compact.
\end{enumerate}
\end{theorem}

\begin{proof}
(a) $\Rightarrow$ (b):
This is Corollary~\ref{cor:convergence to cpt and asymptotic cptness in uniformizable sp}.
(b) $\Rightarrow$ (c):
This is obvious by definition.
(c) $\Rightarrow$ (d):
This is Theorem~\ref{thm:asymptotic cptness and limit set cptness in uniformizable sp}.
(d) $\Rightarrow$ (a): This is obvious.
\end{proof}

This is an extension of Lemma~\ref{lem:eventual Lagrange stability and limit set cptness in locally cpt regular sp} because every locally compact regular space is completely regular, which is equivalent to the uniformizability (ref.\ \cite{Kelley 1955}).

\section{Asymptotic sequential compactness and limit sets}\label{sec:asymptotic seq cptness}

Throughout this section, let $X$ be a topological space.
The purpose of this section is to investigate the sequential versions of the asymptotic compactness and the weak asymptotic compactness introduced in Section~\ref{sec:asymptotic cptness and limit set cptness}.

\subsection{Sequential directed sets}

For the purpose stated above, we introduce the notion of the sequentiality of directed sets as follows.

\begin{definition}\label{dfn:sequential directed set}
We say that a directed set $S$ is \textit{sequential} if there exists a sequence $(s_n)_{n = 1}^\infty$ in $S$ such that $s_n \to \bd(S)$ as $n \to \infty$.
\end{definition}

As shown in the following lemma, we can choose a monotone final sequence in every sequential directed set.
Therefore, we may assume that the sequence $(s_n)_{n = 1}^\infty$ in Definition~\ref{dfn:sequential directed set} is monotone and final.

\begin{lemma}\label{lem:monotone final seq in seq directed set}
Let $S = (S, \le)$ be a directed set and $(t_n)_{n = 1}^\infty$ be a sequence in $S$ with $t_n \to \bd(S)$ as $n \to \infty$.
Then there exists a monotone final sequence $(s_n)_{n = 1}^\infty$ in $S$ such that $s_n \ge t_n$ for all $n \ge 1$.
\end{lemma}

\begin{proof}
We construct a sequence $(s_n)_{n = 1}^\infty$ satisfying
\begin{equation*}
	s_{n + 1} \ge s_n \amd s_n \ge t_n
\end{equation*}
for all $n \ge 1$ by the following induction argument (by the dependent choice precisely):
\begin{enumerate}[label=(\roman*)]
\item Let $s_1 = t_1$.
\item Let $n \ge 1$ be given.
We assume that $s_1, \dots, s_n \in S$ are chosen so that $s_i \ge t_i$ for all $1 \le i \le n$ and $s_1 \le \dots \le s_n$.
Since $S$ is directed, we can choose an upper bound $s_{n + 1} \in S$ of $s_n, t_{n + 1}$.
Then $s_{n + 1}$ satisfies $s_{n + 1} \ge s_n$ and $s_{n + 1} \ge t_{n + 1}$.
\end{enumerate}
Since $t_n \to \bd(S)$ as $n \to \infty$, the sequence $(s_n)_{n = 1}^\infty$ is final.
\end{proof}

\subsection{Sequential limit sets}

In this subsection, we introduce the following sequential limit sets.

\begin{definition}\label{dfn:seq limit set}
Let $S$ be a sequential directed set and $(X_s)_{s \in S}$ be a net of subsets of $X$.
We define a subset $L_\mathrm{seq}(X_s)_{s \in S}$ of $X$ as follows: $y \in L_\mathrm{seq}(X_s)_{s \in S}$ if there exist a subnet of $(X_{s_n})_{n = 1}^\infty$ of $(X_s)_{s \in S}$ and $(y_n)_{n = 1}^\infty \in \prod_{n = 1}^\infty X_{s_n}$ such that the sequence $(y_n)_{n = 1}^\infty$ in $X$ converges to $y$.
We call $L_\mathrm{seq}(X_s)_{s \in S}$ the \textit{sequential limit set} of $(X_s)_{s \in S}$.
\end{definition}

We note that $L_\mathrm{seq}(X_s)_{s \in S} \subset L(X_s)_{s \in S}$ holds from Theorem~\ref{thm:characterization of limit set}.
We now show that the sequential limit set is identical to the limit set when $X$ is first-countable and $S$ is sequential.
We recall that $X$ is said to be \textit{first-countable} if every point in $X$ has a countable local base.

\begin{remark}
Let $x \in X$ be given.
Suppose that $\Set{U_n \in \mathcal{N}_x}{n \ge 1}$ is a countable local base at $x$.
By considering a family $\Set{U'_n}{n \ge 1}$ in $\mathcal{N}_x$ given by
\begin{equation*}
	U'_n := U_1 \cap \dots \cap U_n \mspace{20mu} (n \ge 1),
\end{equation*}
we may assume that $U_n \supset U_{n + 1}$ holds for all $n \ge 1$, i.e., $(U_n)_{n = 1}^\infty$ is a monotone final sequence in $\mathcal{N}_x$.
\end{remark}

The following lemma shows that the concept of subsequences is suffice for first-countable spaces.

\begin{lemma}[ref.\ \cite{Kelley 1955}]\label{lem:cluster pt and subseq in first-countable sp}
Let $x \in X$ be a cluster point of some sequence $(x_n)_{n = 1}^\infty$ in $X$.
If $X$ is first-countable, then $(x_n)_{n = 1}^\infty$ has a subsequence converging to $x$.
\end{lemma}

\begin{lemma}\label{lem:limit set in first-countable sp for sequential directed set}
Let $S = (S, \le)$ be a sequential directed set and $(X_s)_{s \in S}$ be a net of subsets of $X$.
If $X$ is first-countable, then $L(X_s)_{s \in S} = L_\mathrm{seq}(X_s)_{s \in S}$ holds.
\end{lemma}

\begin{proof}
We only have to consider the case $L(X_s)_{s \in S} \ne \emptyset$.
Let $y \in L(X_s)_{s \in S}$ be given.
We choose a monotone final sequence $(s_n)_{n = 1}^\infty$ in $S$ and a monotone final sequence $(U_n)_{n = 1}^\infty$ in $\mathcal{N}_y$.
Then for every $n \ge 1$, $U_n \cap \bigcup_{t \in S, t \ge s_n} X_t \ne \emptyset$ holds.
Therefore, we can choose a sequence $(y_n)_{n = 1}^\infty$ in $X$ so that
\begin{equation*}
	y_n \in U_n \cap \bigcup_{t \in S, t \ge s_n} X_t
\end{equation*}
for all $n \ge 1$.
This implies that $(y_n)_{n = 1}^\infty$ converges to $y$.
The above also implies the existence of a sequence $(t_n)_{n = 1}^\infty$ in $S$ such that $t_n \ge s_n$ and $y_n \in X_{t_n}$ for all $n \ge 1$.
Therefore, $y \in L_\mathrm{seq}(X_s)_{s \in S}$ holds.
\end{proof}

\subsection{Asymptotic sequential compactness}

In this subsection, we introduce the sequential versions of the asymptotic compactness and the weak asymptotic compactness.

\begin{definition}[cf.\ \cite{Sell--You 2002}]\label{dfn:asymptotic seq cptness}
Let $S$ be a directed set and $(X_s)_{s \in S}$ be a net of subsets of $X$.
We say that $(X_s)_{s \in S}$ is \textit{asymptotically sequentially compact} if (i) $S$ is sequential and (ii) for every subnet $(X_{s_n})_{n = 1}^\infty$ of $(X_s)_{s \in S}$ and every $(y_n)_{n = 1}^\infty \in \prod_{n = 1}^\infty X_{s_n}$, the sequence $(y_n)_{n = 1}^\infty$ in $X$ has a convergent subsequence.
\end{definition}

\begin{remark}
Suppose that $S$ is sequential.
If there exists $s_0 \in S$ such that $X_s = \emptyset$ for all $s \ge s_0$, then it holds that $(X_s)_{s \in S}$ is asymptotically sequentially compact.
\end{remark}

\begin{remark}
Let $X = (X, d)$ be a metric space, $\varPhi \colon \R^+ \times X \to X$ be a semiflow, and $E \subset X$ be a subset.
In \cite{Sell--You 2002}, $\varPhi$ is said to be \textit{asymptotically compact on $E$} if for every sequence $(t_n)_{n = 1}^\infty$ in $\R^+$ with $t_n \to \infty$ as $n \to \infty$ and every sequence $(x_n)_{n = 1}^\infty$ in $E$, $(\varPhi(t_n, x_n))_{n = 1}^\infty$ has a convergent subsequence.
The asymptotic sequential compactness generalizes this notion.
\end{remark}

\begin{definition}\label{dfn:weak asymptotic seq cptness}
Let $S$ be a directed set and $(X_s)_{s \in S}$ be a net of subsets of $X$.
We say that $(X_s)_{s \in S}$ is \textit{weakly asymptotically sequentially compact} if (i) $S$ is sequential and (ii) for every monotone final sequence $(s_n)_{n = 1}^\infty$ in $S$ and every $(y_n)_{n = 1}^\infty \in \prod_{n = 1}^\infty \bigcup_{t \in S, t \ge s_n} X_t$, the sequence $(y_n)_{n = 1}^\infty$ in $X$ has a convergent subsequence.
\end{definition}

\begin{remark}
Suppose that $S$ is sequential, $(s_n)_{n = 1}^\infty$ is a monotone final sequence in $S$, and $(X_s)_{s \in S}$ is weakly asymptotically sequentially compact.
Then for every $(y_n)_{n = 1}^\infty \in \prod_{n = 1}^\infty X_{s_n}$, the sequence $(y_n)_{n = 1}^\infty$ in $X$ has a convergent subsequence.
Therefore, $L_\mathrm{seq}(X_s)_{s \in S}$ is nonempty if every $X_s$ is nonempty.
\end{remark}

By definition, the asymptotic sequential compactness implies the weak asymptotic sequential compactness.
It is not apparent whether the converse holds or not.
The following is the sequential version of Lemma~\ref{lem:weak asymptotic cptness and limit set}.

\begin{lemma}\label{lem:asymptotic seq cptness and seq limit set}
Let $S = (S, \le)$ be a sequential directed set and $(X_s)_{s \in S}$ be a net of subsets of $X$.
If $(X_s)_{s \in S}$ is weakly asymptotically sequentially compact, then $(X_s)_{s \in S}$ converges from above to $L(X_s)_{s \in S}$.
\end{lemma}

\begin{proof}
We suppose that $(X_s)_{s \in S}$ does not converge from above to $L(X_s)_{s \in S}$.
Then there exists an open neighborhood $U$ of $L(X_s)_{s \in S}$ such that for all $s \in S$,
\begin{equation*}
	\bigcup_{t \in S, t \ge s} X_t \cap (X \setminus U) \ne \emptyset.
\end{equation*}
Therefore, we can choose a net $(y_s)_{s \in S}$ in $X$ so that
\begin{equation*}
	y_s \in \bigcup_{t \in S, t \ge s} X_t \cap (X \setminus U)
\end{equation*}
for all $s \in S$.
By choosing a monotone final sequence $(s_n)_{n = 1}^\infty$ in $S$, we can obtain
\begin{equation*}
	\emptyset
	\ne L_\mathrm{seq}(X_s)_{s \in S} \cap (X \setminus U)
	\subset L(X_s)_{s \in S} \cap (X \setminus U)
\end{equation*}
in the similar way to Lemma~\ref{lem:weak asymptotic cptness and limit set}.
This is a contradiction.
\end{proof}

In the remainder of this subsection, we study a relation between the convergence from above to a compact set and the asymptotic sequential compactness when $S$ is sequential and $X$ is first-countable.
For this purpose, we use the following lemma.

\begin{lemma}\label{lem:convergence of seq to cpt}
Let $(x_n)_{n = 1}^\infty$ be a sequence in $X$.
If $(x_n)_{n = 1}^\infty$ converges from above to some nonempty compact set $K \subset X$, then it has a convergent subnet.
Consequently, furthermore, if $X$ is first-countable, then $(x_n)_{n = 1}^\infty$ has a convergent subsequence.
\end{lemma}

We give the proof for the sake of completeness although it is standard.

\begin{proof}[Proof of Lemma~\ref{lem:convergence of seq to cpt}]
Let
\begin{equation*}
	\bar{K} := \Set{x_n}{n \ge 1} \cup K.
\end{equation*}
We claim that $\bar{K}$ is compact.
Then the existence of a convergent subnet follows by Theorem~\ref{thm:cptness and net}.
Let $\mathcal{U}$ be an open cover of $\bar{K}$.
Since $\mathcal{U}$ is also an open cover of $K$, there is a finite sub-collection $\mathcal{U}_0 \subset \mathcal{U}$ such that $\mathcal{U}_0$ is a cover of $K$.
Then $\bigcup \mathcal{U}_0 := \bigcup_{U \in \mathcal{U}_0} U$ is a neighborhood of $K$.
Therefore, there is $n_0 \ge 1$ such that for all $n \ge n_0$,
\begin{equation*}
	x_n \in \bigcup \mathcal{U}_0
\end{equation*}
holds.
Since $\Set{x_n}{1 \le n < n_0}$ is a finite set, there is a finite sub-collection $\mathcal{U}_1 \subset \mathcal{U}$ such that $\mathcal{U}_1$ is a cover of $\Set{x_n}{1 \le n < n_0}$.
This shows that $\mathcal{U}_0 \cup \mathcal{U}_1$ is a subcover of $\bar{K}$.
Therefore, the compactness of $\bar{K}$ is obtained.
Finally, the existence of a convergent subsequence under the first-countability of $X$ is a consequence of Lemma~\ref{lem:cluster pt and subseq in first-countable sp}.
\end{proof}

The following is a corollary of Lemma~\ref{lem:convergence of seq to cpt}.

\begin{corollary}\label{cor:convergence to cpt and asymptotic seq cptness}
Suppose that $X$ is first-countable.
Let $S$ be a sequential directed set and $(X_s)_{s \in S}$ be a net of subsets of $X$.
If $(X_s)_{s \in S}$ converges from above to some compact set $K \subset X$, then $(X_s)_{s \in S}$ is asymptotically sequentially compact.
\end{corollary}

\begin{proof}
Let $(X_{s_n})_{n = 1}^\infty$ be a subnet of $(X_s)_{s \in S}$.
We only have to consider the case that $K$ and $\prod_{n = 1}^\infty X_{s_n}$ are nonempty.
Let $(y_n)_{n = 1}^\infty \in \prod_{n = 1}^\infty X_{s_n}$ be given.
By the assumption, the sequence $(y_n)_{n = 1}^\infty$ in $X$ converges from above to some nonempty compact set.
Therefore, $(y_n)_{n = 1}^\infty$ has a convergent subsequence from Lemma~\ref{lem:convergence of seq to cpt}.
This shows that $(X_s)_{s \in S}$ is asymptotically sequentially compact.
\end{proof}

Lemma~\ref{lem:convergence of seq to cpt} also gives the following statement.
It should be compared with Lemma~\ref{lem:asymptotic seq cptness and seq limit set}.

\begin{lemma}\label{lem:convergence to cpt and limit set for sequential S}
Let $S = (S, \le)$ be a sequential directed set and $(X_s)_{s \in S}$ be a net of nonempty subsets of $X$.
If $(X_s)_{s \in S}$ converges from above to some nonempty compact set, then the following statements hold:
\begin{enumerate}[label=\textup{\arabic*.}]
\item $L(X_s)_{s \in S} \ne \emptyset$.
\item $(X_s)_{s \in S}$ converges from above to $L(X_s)_{s \in S}$.
\end{enumerate}
\end{lemma}

\begin{proof}
1. We choose a monotone final sequence $(s_n)_{n = 1}^\infty$ in $S$ and $(x_s)_{s \in S} \in \prod_{s \in S} X_s$.
Then $(x_{s_n})_{n = 1}^\infty$ is a subnet of $(x_s)_{s \in S}$ and converges from above to some nonempty compact set.
From Lemma~\ref{lem:convergence of seq to cpt}, $(x_{s_n})_{n = 1}^\infty$ has a convergent subnet.
Therefore, $L(X_s)_{s \in S}$ is nonempty from Theorem~\ref{thm:characterization of limit set}.

2. We suppose that $(X_s)_{s \in S}$ does not converge to $L(X_s)_{s \in S}$.
Then there exists an open neighborhood $U$ of $L(X_s)_{s \in S}$ such that for all $s \in S$, $\bigcup_{t \in S, t \ge s} X_t \cap (X \setminus U) \ne \emptyset$ holds.
Therefore, we can choose a net $(y_s)_{s \in S}$ in $X$ so that
\begin{equation*}
	y_s \in \bigcup_{t \in S, t \ge s} X_t \cap (X \setminus U)
\end{equation*}
for all $s \in S$.
By choosing a monotone final sequence $(s_n)_{n = 1}^\infty$ in $S$, we can obtain
	\begin{equation*}
		L(X_s)_{s \in S} \cap (X \setminus U) \ne \emptyset
	\end{equation*}
in the same way as the proof of the statement 1.
This is a contradiction.
\end{proof}

The following is a corollary of Lemma~\ref{lem:convergence to cpt and limit set for sequential S}.

\begin{corollary}[cf.\ \cite{Marzocchi--ZandonellaNecca 2002}]\label{cor:limit set cptness in Hausdorff or regular sp}
Suppose that $X$ is Hausdorff or regular.
Let $S = (S, \le)$ be a sequential directed set and $(X_s)_{s \in S}$ be a net of nonempty subsets of $X$.
Then $(X_s)_{s \in S}$ converges from above to some nonempty compact set if and only if $(X_s)_{s \in S}$ is limit set compact.
\end{corollary}

\begin{proof}
We need to show the only-if-part.
Lemma~\ref{lem:convergence to cpt and limit set for sequential S} implies that $L(X_s)_{s \in S}$ is nonempty and $(X_s)_{s \in S}$ converges from above to $L(X_s)_{s \in S}$.
Let $K$ be a nonempty compact set to which $(X_s)_{s \in S}$ converges from above.
Then we have
\begin{equation*}
	L(X_s)_{s \in S} \subset \cls(K)
\end{equation*}
from Lemma~\ref{lem:convergence to cpt in Hausdorff sp} or Lemma~\ref{lem:convergence in regular sp}.
Since $\cls(K)$ is compact from Lemma~\ref{lem:cptness of closure of cpt in regular sp}, this implies the compactness of $L(X_s)_{s \in S}$ by its closedness.
Therefore, $(X_s)_{s \in S}$ is limit set compact.
\end{proof}

\begin{remark}
Suppose that $X$ is a regular Hausdorff space.
Let $\varPhi \colon \R^+ \times X \to X$ be a semiflow and $E \subset X$ be a subset.
In \cite[Theorem 2.10]{Marzocchi--ZandonellaNecca 2002}, it is proved that if $E$ is attracted by some nonempty compact set under $\varPhi$, then the omega limit set $\omega_\varPhi(E)$ of $E$ for $\varPhi$ is a nonempty compact set which attracts $E$ under $\varPhi$.
\end{remark}

\subsection{Asymptotic sequential compactness and limit set compactness in pseudo-metrizable spaces}

As the following theorem shows, in Theorem~\ref{thm:asymptotic cptness and limit set cptness in uniformizable sp}, we can replace the weak asymptotic compactness with the weak asymptotic sequential compactness when $X$ is pseudo-metrizable and $S$ is sequential.

\begin{theorem}[cf.\ \cite{Sell--You 2002}]\label{thm:asymptotic cptness and limit set cptness in pseudo-metrizable sp}
Suppose that $X$ is pseudo-metrizable.
Let $S = (S, \le)$ be a sequential directed set and $(X_s)_{s \in S}$ be a net of subsets of $X$.
If $(X_s)_{s \in S}$ is weakly asymptotically sequentially compact, then $L(X_s)_{s \in S}$ is compact.
Consequently, furthermore, if every $X_s$ is nonempty, then $(X_s)_{s \in S}$ is limit set compact.
\end{theorem}

We give the proof for the sake of completeness although it is similar to the proof of the compactness of the omega limit set $\omega_\varPhi(E)$ of $E$ for a semiflow $\varPhi \colon \R^+ \times X \to X$ given in \cite{Sell--You 2002}.

\begin{proof}[Proof of Theorem~\ref{thm:asymptotic cptness and limit set cptness in pseudo-metrizable sp}]
Let $d$ be a pseudo-metric generating the topology of $X$.
Let $(s_n)_{n = 1}^\infty$ be a monotone final sequence in $S$ and $(z_n)_{n = 1}^\infty$ be a sequence in $L(X_s)_{s \in S}$.
Then for every $n \ge 1$, $B_d(z_n; 1/n) \cap \bigcup_{t \in S, t \ge s_n} X_t \ne \emptyset$.
Therefore, we can choose a sequence $(y_n)_{n = 1}^\infty$ in $X$ so that
\begin{equation*}
	y_n \in B_d {\left( z_n; \frac{1}{n} \right)} \cap \bigcup_{t \in S, t \ge s_n} X_t
\end{equation*}
for all $n \ge 1$.
The sequence $(y_n)_{n = 1}^\infty$ has a subsequence $(y_{n_k})_{k = 1}^\infty$ converging to some $y \in L(X_s)_{s \in S}$ by the weak asymptotic sequential compactness of $(X_s)_{s \in S}$.
Then $(z_{n_k})_{k = 1}^\infty$ also converges to $y$ because
	\begin{equation*}
		d(z_{n_k}, y) \le d(z_{n_k}, y_{n_k}) + d(y_{n_k}, y)
	\end{equation*}
holds for all $k \ge 1$.
Therefore, it holds that $L(X_s)_{s \in S}$ is sequentially compact.
The limit set compactness of $(X_s)_{s \in S}$ under the non-emptiness of every $X_s$ is a consequence of Lemma~\ref{lem:asymptotic seq cptness and seq limit set}.
\end{proof}

We finally obtain the equivalence of the asymptotic sequential compactness and the limit set compactness in pseudo-metrizable spaces when $S$ is sequential.

\begin{theorem}\label{thm:asymptotic seq cptness and limit set cptness in pseudo-metrizable sp}
Suppose that $X$ is pseudo-metrizable.
Let $S$ be a sequential directed set and $(X_s)_{s \in S}$ be a net of nonempty subsets of $X$.
Then the following properties are equivalent:
\begin{enumerate}[label=\textup{(\alph*)}]
\item $(X_s)_{s \in S}$ converges from above to some nonempty compact set.
\item $(X_s)_{s \in S}$ is asymptotically sequentially compact.
\item $(X_s)_{s \in S}$ is weakly asymptotically sequentially compact.
\item $(X_s)_{s \in S}$ is limit set compact.
\end{enumerate}
\end{theorem}

\begin{proof}
(a) $\Rightarrow$ (b):
This follows by Corollary~\ref{cor:convergence to cpt and asymptotic seq cptness} because every pseudo-metrizable space is first-countable.
(b) $\Rightarrow$ (c):
This follows by definition.
(c) $\Rightarrow$ (d):
This follows by Theorem~\ref{thm:asymptotic cptness and limit set cptness in pseudo-metrizable sp}.
(d) $\Rightarrow$ (a):
This is obvious.
\end{proof}

The following corollary is a combination of Theorems~\ref{thm:equivalence between asymptotic cptness and limit set cptness in uniformizable sp} and \ref{thm:asymptotic seq cptness and limit set cptness in pseudo-metrizable sp}.
The proof can be omitted.

\begin{corollary}\label{cor:asymptotic cptness and asymptotic seq cptness in pseudo-metrizable sp}
Suppose that $X$ is pseudo-metrizable.
Let $S$ be a sequential directed set and $(X_s)_{s \in S}$ be a net of nonempty subsets of $X$.
Then $(X_s)_{s \in S}$ is asymptotically compact if and only if it is asymptotically sequentially compact.
\end{corollary}

\section*{Acknowledgments}

This work was supported by the Research Alliance Center for Mathematical Sciences of Tohoku University, the Advanced Institute for Materials Research of Tohoku University, the Research Institute for Mathematical Sciences for an International Joint Usage/Research Center located in Kyoto University, Grant-in-Aid for JSPS Fellows Grant Number JP15J02604, JSPS Grant-in-Aid for Scientific Research on Innovative Areas Grant Number JP17H06460, and JSPS Grant-in-Aid for Young Scientists Grant Number JP19K14565.

\addcontentsline{toc}{section}{Acknowledgments}

\appendix

\end{document}